\documentclass{article}

\usepackage[english]{babel}
\usepackage[letterpaper,top=2cm,bottom=2cm,left=3cm,right=3cm,marginparwidth=1.75cm]{geometry}
\usepackage{amsmath}
\usepackage{mathtools}
\usepackage{graphicx}
\usepackage{amssymb}
\usepackage[colorlinks=true, allcolors=blue]{hyperref}
\usepackage{amsthm}
\usepackage{enumitem}
\usepackage{tikz}
\usepackage{tikz-cd}
\usepackage{todonotes}
\usetikzlibrary{arrows.meta,
                decorations.markings}
\usetikzlibrary{topaths,calc}
\usepackage{pgfplots}
\usepackage{adjustbox}
\usepackage{subcaption} 
\usepackage{caption} 
\usepackage{floatrow}
\usepackage{nicematrix}
\usepackage{authblk}
\usepackage{algorithm}
\usepackage{algpseudocode}
\usepackage{appendix}
\usepackage{comment}

\usepackage{bbm}
\usepackage{bm}

\usepackage[capitalise]{cleveref}
\usepackage{placeins}

\newcommand{\Email}[1]{\ifthenelse{\equal{#1}{}}{}{\par\noindent {\rm E-mail: }{\it  #1} \par}}
\newcommand{\EmailMarked}[1]{\ifthenelse{\equal{#1}{}}{}{\par\noindent $^*$~{\rm E-mail: }{\it  #1} \par}}
\newcommand{\URLaddress}[1]{\ifthenelse{\equal{#1}{}}{}{\par\noindent {\rm URL: }{\tt  #1} \par}}
\newcommand{\URLaddressMarked}[1]{\ifthenelse{\equal{#1}{}}{}{\par\noindent $^*$~{\rm URL: }{\tt  #1} \par}}
\newcommand{\EmailD}[1]{\ifthenelse{\equal{#1}{}}{}{\par\noindent {$\phantom{^{{\rm a)}}}$~\rm E-mail: }{\it  #1} \par}}
\newcommand{\EmailDD}[1]{\ifthenelse{\equal{#1}{}}{}{\par\noindent {$\phantom{{}^{\dag^1}}$~\rm E-mail: }{\it  #1} \par}}

\newtheorem{resultx}{Result}

\theoremstyle{definition}

\newtheorem{theorem}{Theorem}[section]

\newtheorem{definition}[theorem]{Definition}
\newtheorem{remark}[theorem]{Remark}
\newtheorem{proposition}[theorem]{Proposition}
\newtheorem{lemma}[theorem]{Lemma}
\newtheorem{corollary}[theorem]{Corollary}
\newtheorem{example}[theorem]{Example}

\newcommand{\diag}{\mathrm{diag}}

\newcommand{\KK}{\mathcal{K}}
\newcommand{\LL}{\mathcal{L}}

\newcommand{\id}{\mathrm{Id}}
\newcommand{\epsmach}{\epsilon_{\text{m}}}
\date{} 

\pgfplotsset{compat=1.18}

\title{The persistent Laplacian of non-branching complexes}
\author[1]{Magnus Bakke Botnan}
\author[1]{Rui Dong}
\affil[1]{Department of Mathematics, Vrije Universiteit Amsterdam, The Netherlands}

\begin{document}
\maketitle


\begin{abstract}
Non-branching matrices are real matrices with entries in $\{-1,0,1\}$, where each row contains at most two non-zero entries. Such matrices naturally arise in the study of Laplacians of pseudomanifolds and cubical complexes. 
We show that a basis for the kernel of a non-branching matrix can be computed in near-linear time. 
Especially,
the basis has the special property that the supports of all the column vectors in it are disjoint with each other.
Building on this result, we show that the up persistent Laplacian can be computed in near-linear time for a pair of such spaces and its eigenvalues can be more efficiently computed via computing singular values.
In addition to that,
we analyze the arithmetic operations of up persistent Laplacian with respect to a non-branching filtration.
Furthermore, we show that the up persistent Laplacian of $q$-non-branching simplicial complexes can be represented as the Laplacian of an associated hypergraph, thus providing a higher-dimensional generalization of the Kron reduction,
as well as a Cheeger-type inequality.
Finally, we highlight the efficiency of our method on image data.
\end{abstract}

\section{Introduction}
The eigenvectors and eigenvalues of the Laplacian for graphs carry important structural information about the underlying graph and have found numerous applications throughout the sciences \cite{von2007tutorial,lim2020hodge}. Recently, extending these ideas to (filtrations) of simplicial complexes and hypergraphs have become an active topic of research with significant potential. In this paper, our focus lies with the persistent Laplacian  \cite{MR4164275, pers_lap,gen_pers_lap}, which is a generalization of the classical Laplacian to pairs of spaces. Specifically, for a pair of simplicial, or cubical, complexes $\KK \hookrightarrow \LL$, the $q$-th persistent Laplacian $\triangle_q^{\KK, \LL}$ is defined as the sum of the up persistent Laplacian $\triangle_{q, \mathrm{up}}^{\KK, \LL}$ and the standard down Laplacian $\triangle_{q, \mathrm{down}}^{\KK}$ on $\KK$ . It is worth remarking that this allows for a generalized version of the celebrated Hodge theorem \cite{pers_lap}: the nullity of $\triangle_q^{\KK, \LL}$ is equal to the $q$-th persistent Betti number of the pair $\KK \hookrightarrow \LL$. Additionally, the spectra of $\triangle_q^{\KK, \LL}$ encode finer geometric information of a dataset, which cannot be extracted via standard TDA techniques such as persistent homology \cite{davies23c,wei2023}. 

The bottleneck in computing the persistent Laplacian is the computation of $\triangle_{q, \mathrm{up}}^{\KK, \LL}$, for which two algorithms were recently proposed in \cite{pers_lap}, both with cubic time complexity. Specifically, if the number of $q$-simplices and $(q+1)$-simplices of $\LL$ are denoted by $n_q^\LL$ and $n_{q+1}^\LL$, respectively, the time complexity of computing the up persistent Laplacian $\triangle_{q, \mathrm{up}}^{\KK, \LL}$ is
\[
O\left(n_q^\LL (n_{q+1}^\LL)^2 + (n_{q+1}^\LL)^3\right) \qquad \text{or}\qquad 
O\left( (n_q^\LL)^3 \right),
\]
depending on whether Gaussian elimination or the Moore–Penrose pseudoinverse of a matrix is used. For further details, see \cite[Algorithm 3.1, Algorithm 4.1]{pers_lap}, and for a recent survey of persistent Laplacians, we refer to \cite{wei2023}.

\subsection{Overview and contributions}
The goal of this paper is to study the computational and mathematical properties of the persistent Laplacian of $\KK\hookrightarrow \LL$, where $\LL$ is \emph{$q$-non-branching}. That is, any $q$-simplex ($q$-cube) in $\LL$ is the boundary of at most two different $(q+1)$-simplices (($q+1$)-cubes). Equivalently, each row of the $(q+1)$-boundary matrix contains at most two non-zero entries, which are either $1$ or $-1$. We shall refer to matrices with this pattern as \emph{non-branching matrices}. Examples of such complexes are $(q+1)$-pseudomanifolds, e.g., $(q+1)$-dimensional simplicial complexes in $\mathbb{R}^{q+1}$, and $(q+1)$-dimensional cubical complexes. The latter is particularly relevant for the analysis of image data; see \cref{subsec:cubical}. 


In \cref{sec:non_brch_mtx}, we prove the following result:

\begin{resultx}[\cref{thm:weakfactor} and \cref{prop:alg_weak_red}]
    Let $D$ be a non-branching matrix without zero columns. Then, there exist matrices $R \in M_{k \times l}(\mathbb{R})$ and $V \in M_{l \times l}(\mathbb{R})$ computable in $O(k\times\alpha(l)+l)$ arithmetic operations such that $R = DV$, where $\alpha(\cdot)$ is the inverse Ackermann function,
    and $V$ is upper-triangular with entries in $\{-1, 0, 1\}$, the non-zero columns of $R$ are linearly independent,
    {and the supports of the columns in $V$ which have more than one non-zero entries are disjoint to each other.}
    Moreover, these columns correspond precisely to the zero columns of $R$.
\end{resultx}

The key observation is that $D^T$ can be interpreted as the incidence matrix of a graph with loops. 
If $D$ is provided as a sparse matrix, 
we can apply the Disjoint-Union (DSU) algorithm to find the connected components in near-linear time. 
Furthermore, since the rank of $D$ can be immediately derived from $R$, the rank of $D$ is also computable in linear time for sparse inputs. This observation was first made for regular matrices (see Definition~\ref{def:reg_mtrx}) in \cite{twononzero}. While the rank computation in \cite{twononzero} is related to our central strategy, it does not incorporate matrix factorizations as described above. 
By leveraging the theory of graph with loops, our method extends to a broader class of matrices, which is crucial for the computation of the persistent Laplacian.

In \cref{sec:up_pers_lap}, we focus on the efficient computation of the persistent Laplacian and its eigenvalues, 
we show the following two results, where $n_q^\LL$ denotes the number of $q$-simplices in $\LL$, and similarly for the other cases:

\begin{resultx}[\cref{prop:alg_lap}]\label{result:lap_arithm}
    Let $\KK \hookrightarrow \LL$ be simplicial complexes, where $\LL$ is $q$-non-branching. Then, the up persistent Laplacian $\Delta_{q, \mathrm{up}}^{\KK, \LL}$ can be computed in 
    \[
    O\left((n_q^\LL-n_q^\KK)\times \alpha(n_{q+1}^\LL) + n_{q+1}^\LL+(n_q^\KK)^2\right)
    \]
    arithmetic operations,
    {here $\alpha(\cdot)$ is the inverse Ackermann function.}
\end{resultx}
\begin{remark}\label{rem:full-laplacian}
The \emph{$q$-th persistent Laplacian}, $\Delta_q^{\KK, \LL}: C_q^\KK \to C_q^\KK$, is defined as the sum of the up persistent Laplacian $\Delta_{q, \mathrm{up}}^{\KK, \LL}$ and the down Laplacian $\Delta_{q, \mathrm{down}}^{\KK}$ (see \cref{subsec:sp_cmplx}).
Consequently, the ``persistent" aspect is captured exclusively by $\Delta_{q, \mathrm{up}}^{\KK, \LL}$. This makes the up Laplacian the primary focus from a computational perspective, particularly when analyzing how topological features persist.

The down Laplacian $\Delta_{q, \mathrm{down}}^{\KK}$, which depends solely on $\KK$, can be constructed using the matrix representation of the $q$-th boundary operator ($B_q^\KK$), its transpose, and diagonal matrices containing the weights of $q$-simplices and $(q-1)$-simplices, as detailed in \cref{sec:eigenvalues}. Alternatively, $\Delta_{q, \mathrm{down}}^{\KK}$ can be computed directly in quadratic time (in the number of $q$-simplices) as the difference between a diagonal matrix and a generalized adjacency matrix (see, e.g., \cite[SM1]{pers_lap}). In particular, we see that the computational cost, even in the non-branching setting, is dominated by the up persistent part.
\end{remark}

Additionally, for simplicial filtrations, we have the following results.

\begin{resultx}[\cref{prop:time_filtration}]\label{result:filt}
Let $\KK_0\hookrightarrow\KK_1 \hookrightarrow\cdots\hookrightarrow \KK_m\hookrightarrow \LL$ be a simplicial complex filtration where $\LL$ is $q$-non-branching and each $\KK_i$ contains exactly one more $q$-simplex than $\KK_{i-1}$.
Then the sequence of up persistent Laplacians $\{\Delta_{q, \mathrm{up}}^{\KK_i, \LL}\}_{i=0}^m$ can be computed in 
\[
O\left((m+1)\times (n_q^{\KK_m})^2+(m+1)\times n_{q+1}^\LL+(n_q^\LL-n_q^{\KK_m}+m)\times \alpha(n_{q+1}^\LL)\right)
\]
arithmetic operations.
\end{resultx}

The proof of \cref{result:filt} has the following corollary which is important when considering eigenvalues in \cref{sec:eigenvalues}.
\begin{resultx}[\cref{cor:filt}]
    The sequences of matrix pairings $\left\{\big(B_{q+1}^{\LL, \KK_i}, W_{q+1}^{\LL, \KK_i}\big)\right\}_{i=0}^m$ and the matrix products 
$\left\{(W_q^{\KK_i})^{-1/2}B_{q+1}^{\LL, \KK_i} (W_{q+1}^{\LL, \KK_i})^{1/2}\right\}_{i=0}^m$ can be computed in 
\[
O\left((n_q^\LL-n_q^{\KK_m}+m)\times \alpha(n_{q+1}^\LL) + (m+1)\times (n_{q+1}^\LL + n_q^{\KK_m})\right)
\]
arithmetic operations.
\end{resultx}

\setcounter{theorem}{0}

In \cref{subsec:kron_red}, we study the theoretical properties of the up persistent Laplacian for non-branching complexes. It was shown in \cite[Section 4.3]{pers_lap} that when $\KK \hookrightarrow \LL$ are graphs (1-dimensional simplicial complexes), the up persistent Laplacian $\Delta_{0, \mathrm{up}}^{\KK, \LL}$ can be identified as the Laplacian of a weighted graph $\widetilde{\KK}$. This is referred to as the \emph{Kron reduction} \cite{MR3017573}. While this is not generally true for simplicial complexes \cite[Remark 4.12]{pers_lap}, we show in \cref{subsec:kron_red} that $\Delta_{q, \mathrm{up}}^{\KK, \LL}$ coincides with the Laplacian of a weighted oriented hypergraph, provided $\LL$ is $q$-non-branching. We also prove a Cheeger-type inequality for the up persistent Laplacian in \cref{prop:realcheeger}. Intuitively speaking, a pair of non-branching simplicial complexes forms a collection of polyhedra, and the minimal non-zero eigenvalue of $\Delta_{q, \mathrm{up}}^{\KK, \LL}$ is related to the ratio of the area to the volume of each polyhedron.

In \cref{sec:eigenvalues} we discuss our efficient approach to tackling the high computational cost of determining eigenvalues for persistent Laplacians. We detail how leveraging the singular values of related sparse matrices allows for the rapid computation of the top $k$ eigenvalues for both up and down Laplacians.

In \cref{subsec:cubical}, we highlight the efficiency of our algorithm on image data.

The paper concludes with a discussion in \cref{sec:discussion}.

\section{Column reduction of non-branching matrices}\label{sec:non_brch_mtx} 
In this section, we prove the following central theorem for \emph{non-branching matrices}, i.e., real matrices with elements in $\{-1,0,1\}$ and for which every row contains at most two non-zero entries.
For a vector $v$,
we refer to the non-zero values as  \emph{support values} of $v$,
and the corresponding indices as \emph{support indices} of $v$. 

\begin{theorem}
\label{thm:weakfactor}
    Let $D$ be a non-branching matrix without zero columns. 
    Then, one can find matrices $R\in M_{k\times l}(\mathbb{R})$ and $V\in M_{l\times l}(\mathbb{R})$
    such that $R=DV$, where $V$ is upper-triangular with entries in $\{-1,0,1\}$, the non-zero columns of $R$ are linearly independent, 
    and the support indices of columns in $V$ are mutually disjoint. Furthermore, the $i$-th column of $V$ has more than one non-zero entry if and only if the $i$-th column of $R$ is zero.
\end{theorem}
We refer to such a matrix factorization as a \emph{weak column reduction} of $D$. The matrix $R$ is not \emph{column reduced} in the sense of the ``standard algorithm'' \cite[Ch.~VII.1]{edelsbrunner2010computational} as multiple columns may have their pivots (lowest non-zero entry) at the same row. Note that the columns of $V$ with more than one non-zero entry constitute an orthogonal basis for the kernel of $D$.

It will be useful to decompose a non-branching matrix $D$ into its \emph{orientable} and \emph{non-orientable} components. The following definition is inspired by the structure of the $k$-th boundary map of an oriented $k$-pseudomanifold. 
\begin{remark}
In this section, we focus on matrices with real-valued entries. However, \cref{thm:weakfactor} extends verbatim to arbitrary unital rings that admit constant-time arithmetic operations. In fact, for rings of characteristic~2, the arguments simplify, as every non-branching matrix can be considered oriented.
\end{remark}
\subsection{Oriented matrices}\label{subsec:non_branching_def}


\begin{definition}\label{def:orit_mtrx}
Let $D\in M_{k\times l}(\mathbb{R})$ be a non-branching matrix.
We say $D$ is \emph{oriented} if each row in $D$ contains at most one $1$ and at most one $-1$.
We say $D$ is \emph{orientable} if it can be made oriented by switching the signs of columns. Otherwise, $D$ is \emph{non-orientable}.
\end{definition}

\begin{example}
Consider the following matrices:
\[
D_1
=
\begin{bmatrix}
1 & 0  & -1 \\
1 & -1  & 0 \\
0 & -1 & 1 \\
0 & 0 & -1\\
\end{bmatrix},
\quad
D_2
=
\begin{bmatrix}
1 & 0  & 0 \\
1 & 1  & 0 \\
0 & -1 & -1 \\
0 & 0 & 1\\
\end{bmatrix},
\quad
D_3
=
\begin{bmatrix}
1 & 0  & -1 \\
1 & 1  & 0 \\
0 & 1 & -1 \\
0 & 0 & 1\\
\end{bmatrix}.
\]
The matrix $D_1$ is oriented. 
The matrix $D_2$ is not oriented but is orientable since we can convert $D_2$ to be an oriented matrix by switching the sign of the second column of $D_2$,
i.e.,
\[
D_2
=
\begin{bmatrix}
1 & 0  & 0 \\
1 & 1  & 0 \\
0 & -1 & -1 \\
0 & 0 & 1\\
\end{bmatrix}
\xlongrightarrow{\text{multiply the $2$nd column by $-1$}}
\begin{bmatrix}
1 & 0  & 0 \\
1 & -1  & 0 \\
0 & 1 & -1 \\
0 & 0 & 1\\
\end{bmatrix}.
\]
The matrix $D_3$ is non-orientable.
\end{example}

\begin{definition}
\label{def:reg_mtrx}
We say a non-branching matrix $D\in M_{k\times l}(\mathbb{R})$ is 
\begin{enumerate}
    \item \emph{Non-row-singular} if each row in $D$ contains either two or zero non-zero entries,
and we say $D$ is \emph{row-singular} if there is some row in $D$ with exactly one non-zero entry.
    \item \emph{Regular} (resp. \emph{regulable}) if it is both non-row-singular and oriented (resp. \emph{orientable}).
\item \emph{Irregular} if it is both non-row-singular and non-orientable.
\end{enumerate}
\end{definition}

\begin{example}
Consider the matrices $D_4$ and $D_5$ as follows.
\[
D_4
=
\begin{bmatrix}
1 & 0  & -1 \\
1 & 1  & 0 \\
0 & -1 & -1 \\
\end{bmatrix},
\quad
D_5
=
\begin{bmatrix}
1 & 0  & -1 \\
1 & 1  & 0 \\
0 & 1 & -1 \\
\end{bmatrix}.
\]
Both $D_4$ and $D_5$ are non-row-singular,
since each row in $D_4$ and $D_5$ contains $2$ non-zero entries.
Furthermore, $D_4$ is regulable since it can be oriented by switching the sign of the second column:
\[
D_4
=
\begin{bmatrix}
1 & 0  & -1\\
1 & 1  & 0 \\
0 & -1 & -1 \\
\end{bmatrix}
\xlongrightarrow{\text{multiply the $2nd$ column by $-1$}}
\begin{bmatrix}
1 & 0  & -1 \\
1 & -1  & 0 \\
0 & 1 & -1 \\
\end{bmatrix}.
\]
On the other hand, $D_5$ is irregular.
\end{example}

\subsection{Regular matrices}\label{subsec:reg_mtx}
In this subsection,  $D$ is assumed to be a regular matrix (\cref{def:reg_mtrx}). Without loss of generality, we shall assume that $D$ has a non-zero entry in every column and every row. In particular, since every column contains precisely two non-zero entries with opposite signs, its transpose $D^T$ is the incidence matrix of a directed graph $G$. The following fundamental result from algebraic graph theory will be important; see e.g. \cite[Theorem 8.3.1]{MR1829620} for a proof.

\begin{lemma}\label{lemma:grph_cmpnt}
Let $D\in M_{k\times l}(\mathbb{R})$ be a regular matrix (without zero columns or zero rows). The transpose matrix $D^T$ can be regarded as an incidence matrix of a directed graph $G$,
which has $(l-\mathrm{rank}(D))$ weakly connected components\footnote{ A weakly connected component in a directed graph is a connected component in the underlying undirected graph.}.
\end{lemma}

\begin{remark}
    One can also consider $D^T$ as the $1$-st boundary operator in simplicial homology with real coefficient, in which case it follows that $\beta_0(G) = l - \mathrm{rank}(D^T)$, and $\beta_0$ measures precisely the number of weakly connected components. 
\end{remark}

Let us introduce the following notation. We use the symbol $[k]$ to denote the index set $\{1, 2, \cdots ,k\}$.
For $\emptyset\neq I\subset [k]$ and $\emptyset \neq J \subset [l]$,
we let $D(I, J)$ denote the submatrix of $D$ consisting of the rows and columns indexed by $I$ and $J$, respectively. 
We use $D(:, J)$ and $D(I, :)$ to denote $D([k], J)$ and $D(I, [l])$, respectively, 
and $m$:$n$ denotes the indices $\{m, m+1,\cdots, n\}$ when $m<n$. 
We also use $D(:, -1)$ and $D(-1, :)$ to denote the last column and last row of $D$ respectively.
Let $G$ denote the graph associated to $D^T$, and denote its weakly connected components $G_i$ for $1\leq i\leq r$. For a connected component $G_i$, let $C_i$ denote the column indices in $D$ corresponding to the vertices in $G_i$. That is, $(D(:, C_i))^T$ is the incidence matrix associated with $G_i$ for each $1\leq i\leq r$. Furthermore, for two sets of column indices $C_i$, $C_j$,
we denote by $C_i\#C_j$ the concatenation of $C_i$ and $C_j$,
i.e.,
$C_i\#C_j:=(c_i^{1}, \cdots, c_i^{l_i}, c_j^1, \cdots, c_{j}^{l_j})$,
here $C_i=(c_i^1,\cdots, c_i^{l_i})$, 
and similarly for $C_j$.

\begin{lemma}\label{lem:simpleprod}
    The sum of column vectors $\sum_{j\in C_i} D(:,j)=0$ 
    for each $C_i$ with $1\leq i\leq r$.

\end{lemma}

\begin{proof}
    Since every row contains precisely one +1 and one -1, the result is immediate.
\end{proof}




\begin{lemma}\label{lem:Ul}
Let $U_l$ be the $l\times l$ upper-triangular matrix with $1$'s on the diagonal, and with $U_l(j, c_i^{l_i}) =1$ for all $j\in C_i$. 
    Then, the zero columns of $DU_l$ are given by the following column indices $I = \{c_1^{l_1}, \ldots, c_r^{l_r}\}$. Furthermore, the non-zero columns of $DU_l$ are linearly independent. 
\end{lemma}
\begin{proof}
That $(DU_l)(:,I)$ is the zero matrix, is immediate from \cref{lem:simpleprod}. Furthermore, note that $(DU_l)(:,[l]\setminus I) = D(:, [k]\setminus I)$. Hence, the columns in $I$ are precisely the zero columns. According to \cref{lemma:grph_cmpnt},
the rank of $D$ is equal to $(l-r)$. 
We conclude that
all the non-zero columns of $DU_l$ are linearly independent.
\end{proof}
\begin{example}
\label{ex:LemmaU}
Consider the following regular matrix $D$:
\[D
=
\begin{bNiceMatrix}[
  first-row,code-for-first-row=\scriptstyle,
  first-col,code-for-first-col=\scriptstyle,
]
& [1] & [2] & [3] & [4] & [5] & [6] & [7]\\
[13] & 1 & 0 & -1 & 0 & 0 & 0 & 0\\
[35] & 0 & 0 & 1 & 0 & -1 & 0 & 0\\
[42] & 0 & -1 & 0 & 1 & 0 & 0 & 0\\
[51] & -1 & 0 & 0 & 0 & 1 & 0 & 0\\
[46] & 0 & 0 & 0 & 1 & 0 & -1 & 0\\
[47] & 0 & 0 & 0 & 1 & 0 & 0 & -1\\
\end{bNiceMatrix}.\] 
We observe that the associated directed graph $G$ has two connected components, and the corresponding sets of vertices are $C_1=\{1, 3, 5\}$ and $C_2=\{2, 4, 6, 7\}$; see \cref{fig_example}. From \cref{lem:Ul}, we obtain 
\[
U_7
=
\begin{bmatrix}
1 & 0 & 0 & 0 & 1 & 0 & 0\\
0 & 1 & 0 & 0 & 0 & 0 & 1\\
0 & 0 & 1 & 0 & 1 & 0 & 0\\
0 & 0 & 0 & 1 & 0 & 0 & 1\\
0 & 0 & 0 & 0 & 1 & 0 & 0\\
0 & 0 & 0 & 0 & 0 & 1 & 1\\
0 & 0 & 0 & 0 & 0 & 0 & 1\\
\end{bmatrix},\qquad\qquad
R=DU_7=
\begin{bmatrix}
1 & 0 & -1 & 0 & 0 & 0 & 0\\
0 & 0 & 1 & 0 & 0 & 0 & 0\\
0 & -1 & 0 & 1 & 0 & 0 & 0\\
-1 & 0 & 0 & 0 & 0 & 0 & 0\\
  0 & 0 & 0 & 1 & 0 & -1 & 0\\
 0 & 0 & 0 & 1 & 0 & 0 & 0\\
\end{bmatrix}.\] 
\end{example}
\begin{figure}[H]
    \begin{tikzpicture}[
       decoration = {markings,
                     mark=at position .5 with {\arrow{Stealth[length=2mm]}}},
       dot/.style = {circle, fill, inner sep=1.6pt, node contents={},
                     label=#1},
every edge/.style = {draw, postaction=decorate}
                        ]
    \def\side{2.251666}





    \node (1) at (1/2 * \side,  \side * 0.8) [dot=$1$];
    \node (2) at (3/2 * \side,  \side * 0.8) [dot=$2$];
    \node (6) at (5/2 * \side,  \side * 0.8) [dot=$6$];
    \node (3) at (0, -0.3) [dot=$3$];
    \node (5) at (\side, -0.3) [dot=$5$];
    \node (7) at (2 * \side, -0.3) [dot=$7$];
    \node (4) at (2 * \side, 0.45 * \side) [dot=$4$];


    \path (1) edge (3);
    \path (3) edge (5);
    \path (5) edge (1);
    
    \path (4) edge (2);
    \path (4) edge (6);
    \path (4) edge (7);
    
\end{tikzpicture}
\caption{The directed graph $G$  associated to the incidence matrix $D^T$ in \cref{ex:LemmaU}.}
\label{fig_example}
\end{figure}

We say that a square matrix $E$ is a \emph{flag matrix} if $E$ is a diagonal matrix with diagonal entries in $\{-1,1\}$. Summarized, we have the following proposition.

\begin{proposition}
\label{prop:regular}
    Let $D$ be a $k\times l$ regulable matrix. Then, $R=DV$ is a weak column reduction, where $V = EU_l$ for $E$ a flag matrix such that $DE$ is oriented. 
\end{proposition}

\subsection{Irregular matrices}\label{subsec:irreg_mtx}

In the case that $D\in M_{k\times l}(\mathbb{R})$ is irregular as defined in \cref{def:reg_mtrx},
the matrix $D^T$ can no longer be regarded as the incidence matrix of any directed graph. However, 
if we take the absolute value of all entries in $D$, denoting the resulting matrix by $|D|$,
then $\left(|D|\right)^T$ is the incidence matrix of an \emph{undirected} graph $G$ with vertices given by the columns of $D$ as in the previous section. The incidence matrix for a graph with loops is given in \cref{def:incd_mtx_undir}.

\begin{lemma}\label{lemma:no_ort_cmpt}
Let $D\in M_{k\times l}(\mathbb{R})$ be an irregular matrix.
Suppose the matrix $D$ has no zero-columns and zero-rows.
Then, the matrix $D$ has full column rank if the undirected graph $G$ associated with $\left(|D|\right)^T$ has only one connected component.
\end{lemma}

\begin{proof}
Suppose $D$ does not have full column rank. Then,
there is a non-zero vector $v=(v_1,\cdots, v_l)^T\in \mathbb{R}^l$ such that $Dv=0$. 
Notice that the vertices corresponding to columns  $j_1$ and $j_2$ are connected in $G$ if and only if there exists some $i$ such that $D(i,j_1)D(i,j_2) \neq 0$. 
Assume that $v_1=a$. This forces that $v_j$ is either $a$ or $-a$, where $j$ is any vertex connected to vertex $v_1$ in $G$. Since $G$ is connected, we must have that $v\in \{\pm a\}^l$. We may assume that $a=1$. 
Thus there exists a diagonal matrix $E$ with diagonal entries being either $1$ or $-1$ such that the product $Ev=\boldsymbol{1}_l$,
and therefore $Dv=DE^{-1}\boldsymbol{1}_l=0$.
However, this contradicts that $D$ is irregular and therefore non-orientable. 
\end{proof}

\begin{remark}
    The $q$-th boundary matrix of a connected non-orientable $q$-dimensional pseudomanifold $M$ is irregular. We recover the well-known identity $H_q(M; \mathbb{R}) = 0$. So for instance, $H_2(\mathbb{RP}^2;\mathbb{R}) = 0.$
\end{remark}

\subsection{Row-singular matrices}\label{subsec:singular_mtx}
If $D \in M_{k \times l}(\mathbb{R})$ is row-singular, 
we can interpret $(|D|)^T$ as the incidence matrix of an undirected graph with loops,
recall that $|D|=(|d_{ij}|)$ for $D=(d_{ij})$.



\begin{lemma}\label{lemma:sig_cmpt}
Let $D\in M_{k\times l}(\mathbb{R})$ be a row-singular matrix without zero-columns and zero-rows.
Let $G$ be the undirected graph with loops whose incidence matrix is equal to $|D|^T$.
The matrix $D$ has full column rank if the associated graph $G$ has only one connected component.
\end{lemma}

\begin{proof}
To prove that $D$ has full column rank,
it is enough to prove that $Dv=0$ if and only if $v=(v_1,\cdots, v_l)\in R^l$ is the zero vector.
Since the undirected graph $G$ is connected by assumption,
we obtain that 
 $|v_1|=\cdots=|v_l|$ as in the proof of \cref{lemma:no_ort_cmpt}.
However,
since $D$ is row-singular,
there must be some row containing only one non-zero entry, and thus $v_j=0$ for some $j$. It follows that $v=(v_1,\cdots, v_l)=0$.
\end{proof}

\begin{remark}
    When $D$ is row-singular, each edge in $G$ contains either one or two vertices. 
    we see that the loops in $G$ will not be relevant in the computation of connected components. Therefore, we can ignore the loops, and find the connected components of $G$ by running any algorithm for computing connected components in graphs.
\end{remark}

\subsection{Wrapping up the proof of  \cref{thm:weakfactor}}\label{subsec:non_branching_mtx}
To prove \cref{thm:weakfactor}, we first regard $|D|^T$ as the incidence matrix of an undirected graph $G$ (with loops).
Decomposing $G$ into its connected components, we write 
\[
G = \left(\bigcup_{i=1}^r G_i\right) \bigcup \left(\bigcup_{i=1}^q \bar{G}_i\right)\bigcup \left(\bigcup_{i=1}^s \widetilde{G}_i\right),
\]
and correspondingly for the column indices
\[
C = \left(\bigcup_{i=1}^r C_i\right) \bigcup \left(\bigcup_{i=1}^q \bar{C}_i\right)\bigcup \left(\bigcup_{i=1}^s \widetilde{C}_i\right).\]
That is, $D(:,C_i)$ is the incidence matrix for $G_i$, and likewise for the other components. The notation is chosen such that $D(:, C_i)$ is a regulable matrix, $D(:, \bar{C}_i)$ is irregular, and $D(:, \widetilde{C}_i)$ is row-singular. As above, we let $(c_i^{1}, \ldots, c_i^{l_i})$ denote the ordered list of column indices in $C_i$. 
Collecting all our results, we now prove the following theorem, from which \cref{thm:weakfactor} is an immediate corollary. 
\begin{theorem}\label{thm:main_thm}
Let $D\in M_{k\times l}(\mathbb{R})$ be a non-branching matrix without zero columns. Then,
$R = D(EU)$
is a weak column reduction of $D$, where $E$ is a flag matrix such that $(DE)(:,C_i)$ is oriented (and thus regular) for each regulable component $C_i$, and $U$ is the upper-triangular matrix with $1$'s on the diagonal and $U(c_i^{j}, c_i^{1}) = 1$ for all $1\leq i \leq r$.
\end{theorem}
\begin{proof}
Observe that if $\widehat{G}_i$ and $\widehat{G}_j$ are any two components of $G$, then there is no row $k$ such that $D(k, \widehat{C}_i) \neq 0$ and $D(k, \widehat{C}_j) \neq 0$. Hence, columns coming from different components of $G$ are automatically linearly independent. By \cref{lemma:no_ort_cmpt} and \cref{lemma:sig_cmpt}, the restriction of $D$ to the irregular and row-singular components is of maximal column rank, and therefore need no further reduction. It now follows from \cref{lemma:sig_cmpt}, that the stated decomposition is a weak column reduction.

\end{proof}

\subsection{Algorithms and running times}
In this section we analyze the time complexity of computing weak column reduction.
Since a non-branching matrix $D$ is sparse,
we assume that the input matrix $D$ is in a format of sparse matrix.
There are several different options, 
such as
Compressed Sparse Column (CSC),
Compressed Sparse Row (CSR),
and Coordinate (COO).

The key algorithm in computing weak column reduction with respect to a $k\times l$ non-branching matrix $D$ is to find connected components, with input being an incidence matrix.
There are multiple algorithms that can achieve this,
such as depth-first-search (DFS),
Breadth-First Search (BFS),
and Union-Find (Disjoint Set Union, DSU).
We refer to \cite{cormen2022introduction} for more details about graph algorithms.
In this work we apply DSU based on the following reasons:
Firstly,
our input data is an incidence matrix with respect to a graph,
by applying DFS or BFS one has to first convert an incidence matrix to adjacency lists,
which will take $O(kl)$,
while DSU has near-linear time $O(k \, \alpha (l))$,
here $\alpha (l)$ is the inverse Ackermann function,
an extremely slow-growing function.
\footnote{For any practical values of $l$,
$\alpha(l)$ never exceeds $5$.
}
Secondly,
DSU excels for dynamic graphs,
i.e.,
edges and loops varies,
this happens when computing boundary matrices with respect to a non-branching simplicial filtration. This will become important in \cref{sec:filtration}.

Suppose the input $D\in M_{k\times l}(\mathbb{R})$ is a non-branching sparse (CSR format) matrix.
Recall that $|D|$ is the matrix whose non-zero entries are the absolute values of those in $D$.
It is easy to see that the computation of $|D|$ from $D$ takes $O(k)$.
We consider $|D|$ as the transpose of an incidence matrix with respect to an undirected graph with loops,
that is,
we regard each column in $|D|$ as a vertex,
and each row as an edge or a loop.
We first apply DSU (see \cref{alg:dsu} in \cref{appdx_codes}) to find all components,
which takes $O(k\alpha(l))$,
recall that in DSU algorithm, 
two operations \textsc{Union} (see \cref{alg:union} in \cref{appdx_codes}) and
\textsc{Find} (see \cref{alg:find} in \cref{appdx_codes}) are used to merge two sets and find the root of a component, respectively,
and the components are stored in the list ``parent'' that contains the root of each vertex. 
We then need to find the vertices having self-loops and their corresponding roots,
when $|D|$ is stored in the CSR format,
it takes $O(k)$.
\footnote{This can be done by scanning over the CSR format index pointer.}
Next we need to filter out these roots, 
which takes $O(k+l)$ if we traverse the list of parents once and use a hash map to track indices for all roots.

We then need to reorient a non-branching matrix so that filter out all the irregular components.
The main idea is to check if we can ensure that each entry in the column sum is not equal to $2$ and $-2$ by flipping the column signs of $D$.
If this is impossible,
then the matrix $D$ is non-orientable.
We initialize the flag list to be $l$ copies of $1$s,
and scan over rows of $D$,
if the column sum is $2$ or $-2$,
we flip the flag list.
The \textsc{Reorientation} algorithm is shown in \cref{alg:reort}. 
Since $D$ is stored in CSR format,
finding the column indices of non-zero entries in each row takes $O(1)$,
hence the time complexity of \textsc{Reorientation} is $O(k)$.

We summarize the time complexity analysis in the following proposition.

\begin{proposition}
\label{prop:alg_weak_red}
Let $D\in M_{k\times l}(\mathbb{R})$ non-branching matrix.
The time complexity of weak column reduction \cref{alg:weak_red} is $O(k\times \alpha(l) + l)$.
\end{proposition}

\begin{proof}
Recall that \textsc{DSU} takes $O(k\times \alpha({l}))$,
and the procedure of searching roots with respect to loops takes $O(l)$ since the size of vertices is $l$.
Retrieving the column partition $\{I_r\}_{r}$ and the corresponding row partition $\{J_r\}_r$ takes $O(l)$ and $O(k)$, respectively,
computing $\textsc{\text{Reorientation}}(D_r)$ takes $O(|J_r|)$,
and updating $E(I_r, I_r)$ takes $O(I_r)$,
therefore,
it takes $O(k\times \alpha(l)+l)$.
\end{proof}

\begin{remark}
The following shows that a DSU based approach is faster than simply applying Gaussian column reduction. Consider the $k\times (k+1)$ matrix 
\[
D = 
\begin{bmatrix}
1 & -1 & 0  & 0 & \cdots & 0\\
1 & 0 & -1 & 0 & \cdots & 0\\
1 & 0 & 0  & -1 & \cdots & 0\\
\vdots & \vdots & \vdots & \vdots & \ddots & \vdots\\
1 & 0 & 0 & 0 & \cdots & -1
\end{bmatrix},
\]
and observe that the Gaussian column elimination of $D$ needs to add two columns iteratively $k$ times,
which will cost $O(k^2)$. In contrast,
the weak column reduction takes a near-linear time according to \cref{prop:alg_weak_red}.
\end{remark}

\begin{algorithm}[H]
\caption{\textsc{Reorientation}}
\label{alg:reort}
\begin{algorithmic}[1]
\State \textbf{Input:} $D$
\Comment{$D$ is a $k\times l$ sparse non-branching matrix}
\State \textbf{Output:} isorientable, $\textrm{flag}$ 
\Comment{$E$ records the signs, $D'$ is the reorientation of $D$}


\State $\textrm{isorientable}\gets \textrm{true}$
\State $\text{flag}\gets [1, 1, \cdots, 1]$ 
\Comment{Initialize $\text{flag}$ to be a length $l$ list of all $1$s}
\State $\text{vertices} \gets [1, 2, \cdots, l]$ 

\For{$i=1$ to $k$}
\State $\text{edge}_i\gets \text{column index (or indices) of non-zero entry (or entries) in } D(i, :)$ 

\If{$\text{edge}_i$ has two different vertices}
\State $u, v\gets \text{edge}_i$
\If{$\text{flag}[u]*D(i, u)+\text{flag}[v]*D(i, v)=-2 \text{ or }2$}

\If{$u\in \text{vertices}$ and $v\in \text{vertices}$}
\State $D(:, v)\gets -D(:, v)$
\State $\textrm{flag}[v]\gets -1$

\ElsIf{$u\notin \text{vertices}$ and $v\in \text{vertices}$}
\State $D(:, v)\gets -D(:, v)$
\State $\textrm{flag}[v]\gets -1$

\ElsIf{$u\in \text{vertices}$ and $v\notin \text{vertices}$}
\State $D(:, u)\gets -D(:, u)$
\State $\textrm{flag}[u]\gets -1$

\ElsIf{$u\notin \text{vertices}$ and $v\notin \text{vertices}$}
\State $\text{isorientable}\gets \text{false}$
\State \textbf{break}
\Comment{The matrix $D$ is non-orientable, break for loop}
\EndIf

\EndIf
\EndIf
\State $\text{vertices}\gets \text{vertices}\backslash \text{edge}_i$

\EndFor

\end{algorithmic}
\end{algorithm}


\begin{algorithm}[H]
\caption{\textsc{WeakColumnReduction}}
\label{alg:weak_red}
\begin{algorithmic}[1]
\State \textbf{Input:} $D$
\Comment{$D$ is a $k\times l$ matrix}
\State \textbf{Output:} $R$, $E$, $V$
\Comment{$R=DEV$}
\State $R\gets D$
\State $V\gets\id_l$
\State $E\gets \id_l$
\State $\text{parent}\gets \textsc{DSU}(|D|)$
\Comment{parent} is the list of roots of each vertex
\State $\text{roots}\gets$ the set of unique values of parent
\State $S_{\text{loops}}\gets$ the set of vertices having loops
\Comment{$S_{\text{loops}}$ is computed from index pointer of $D$}
\State $\text{roots}_{\text{loops}}\gets \text{parent}[S_\text{loops}]$
\Comment{Find the roots of components having loops}
\State $\text{roots}_{\text{noloops}}\gets \text{roots}\backslash \text{roots}_{\text{loops}}$

\For{$r\in \text{roots}_{\text{noloops}}$}
\Comment{update $R$, $E$ and $V$}
\State $I_r\gets$ the indices of $r$ in parent, sorted in increasing order
\State $J_r\gets$ the indices of non-zero rows in $D(:, I_r)$ 
\State $D_r\gets D(J_r, I_{r})$
\State isorientable, flag $\gets$ \textsc{Reorientation}($D_r$)
\If{isorientable is true}
\State $i_r\gets$ the last element in $I_r$
\State $R(:, i_r)\gets 0$
\Comment{Set the column $R(:, i_r)$ to be zeros}
\State $V(I_r, i_r)\gets 1$
\Comment{Set the $I_r$ rows in column $V(:, i_r)$ to be ones}
\State $E(I_r, I_r)\gets \text{flag}$
\EndIf
\EndFor

\Statex

\end{algorithmic}
\end{algorithm}

\section{A faster algorithm for the up persistent Laplacian}\label{sec:up_pers_lap}
In this section, we leverage our results from \cref{sec:non_brch_mtx} to compute the up persistent Laplacian. We start by defining non-branching simplicial complexes and recalling the definition of up persistent Laplacians as introduced in \cite{pers_lap}. In \cref{subsec:regular_form}, we present the main result of this section, \cref{prop:pers_lap}, which establishes that the computation of the up persistent Laplacian can be achieved using quadratic arithmetic operations for non-branching simplicial complexes.

\subsection{Background material}\label{subsec:sp_cmplx}

Let $\LL$ be a simplicial complex.
We denote by $S_{q}^\LL$ the set of all $q$-simplices in $\LL$.

\begin{definition}
We say a $q$-simplex in $\LL$ is \emph{non-branching} if it is a face of at most two $(q+1)$-simplices,
and we say $\LL$ is \emph{$q$-non-branching} if all the $q$-simplices $S_q^\LL\subset \LL$ are \emph{non-branching}.
\end{definition}
When the dimension is clear from context, we simply refer to $\LL$ as \emph{non-branching}. 
\begin{remark}
    
Observe that if $\LL$ is a $(q+1)$-dimensional pseudomanifold \cite{MR575168}, then
the collection $\LL$ is non-branching.
This is also the case for cubical complexes \cite{MR3025945},
since each $q$-cube is the face of at most two $(q+1)$-cubes.
\end{remark}

Choose an orientation for each simplex $\sigma \in S_q^\LL$,
and denote the set of oriented $q$-simplices as $\bar{S}_q^\LL$.
Let $C_q^\LL$ be the chain group over $\mathbb{R}$ generated by $\bar{S}_q^\LL$,
and denote by $\partial_q^\LL\colon C_q^\LL\to C_{q-1}^\LL$ the boundary operator. We associate to each simplex a positive weight via the function $w^\LL\colon \LL\to \mathbb{R}^+$, and we say $\LL$ is \emph{unweighted} if $w^\LL\equiv 1$. Moreover, we let $w_q^\LL$ denote the restriction of $w^\LL$ to $C_q^\LL$. An inner product over $C_q^\LL$ can be defined by
\[
\langle [\sigma_1], [\sigma_2]\rangle_{w_q^\LL}
\coloneqq
\delta_{\sigma_1\sigma_2}\cdot \left(w_q^\LL(\sigma_1)\right)^{-1},
\]
where $\delta_{\sigma_1\sigma_2}$ is equal to $1$ if and only if $\sigma_1=\sigma_2$,
otherwise it is equal to $0$.
We denote by $\left(\partial_q^\LL\right)^\ast: C_{q-1}^\LL\to C_q^\LL$ the adjoint of $\partial_q^\LL$ with respect to this inner product.

\begin{definition}
\label{total:lap}    
The \emph{$q$-th up Laplacian} is $\triangle^{\LL}_{q,\mathrm{up}}\coloneqq \partial_{q+1}^{\LL}\circ \left(\partial_{q+1}^{\LL}\right)^\ast$,
the \emph{$q$-th down Laplacian} is
$\triangle^{\LL}_{q, \mathrm{down}}\coloneqq
\left(\partial_q^\LL\right)^\ast\circ\partial_q^\LL
$,
and the \emph{$q$-th combinatorial Laplacian} $\triangle_q^\LL: C_q^\LL\to C_q^\LL$ is defined as 
$
 \triangle_q^{\LL}\coloneqq
 \triangle^{\LL}_{q,\mathrm{up}} + 
 \triangle_{q, \mathrm{down}}^\LL.
$
\end{definition}

For a pair of simplicial complexes $\KK\hookrightarrow \LL$, let $C_q^{\LL, \KK}$ be the subspace of $C_q^\LL$ defined as:
\[
C_q^{\LL, \KK}:= \{c\in C_q^\LL: \partial_q^\LL(c)\in C_{q-1}^\KK\}\subset C_q^\LL,
\]
and let $\partial_{q}^{\LL, \KK}:= \partial_q^\LL\Big |_{C_{q}^{\LL, \KK}}$ be the boundary operator restricted to $C_q^{\LL, \KK}$.

\begin{definition}[\cite{pers_lap}]
\label{def:perslap}
The \emph{$q$-th up persistent Laplacian} $\triangle_{q, \mathrm{up}}^{\KK, \LL}: C_q^\KK\to C_q^\KK$ is 
\[
 \triangle_{q, \mathrm{up}}^{\KK, \LL}\coloneqq \partial_{q+1}^{\LL, \KK}\circ \left(\partial_{q+1}^{\LL, \KK}\right)^\ast,
\]
and the \emph{$q$-th persistent Laplacian} $\triangle_{q}^{\KK, \LL}: C_q^\KK\to C_q^\KK$ is
\[
 \triangle_{q}^{\KK, \LL}\coloneqq \partial_{q+1}^{\LL, \KK}\circ \left(\partial_{q+1}^{\LL, \KK}\right)^\ast +
 \left(\partial_q^\KK\right)^\ast \circ \partial_q^\KK.
\]
\end{definition}

\begin{figure}
\centering
\begin{tikzcd}
\centering
     C_{q+1}^\KK\arrow{rr}{\partial_{q+1}^\KK}\arrow[hookrightarrow,dashed,gray]{dd} && C_q^\KK\arrow[rr,shift left=.75ex,blue,"\partial_q^\KK"]\arrow[hookrightarrow,dashed,gray]{dd}\arrow[dl,blue,"\left(\partial_{q+1}^{\LL, \KK}\right)^*" {yshift=-2pt, xshift=-7pt}] && C_{q-1}^\KK\arrow[hookrightarrow,dashed,gray]{dd}\arrow[ll,shift left=.75ex,blue,"\left(\partial_q^\KK\right)^*"]\\
      &C_{q+1}^{\LL, \KK}\arrow[hookrightarrow,dashed,gray]{dl}\arrow[ur,shift left=.75ex,blue,"\partial_{q+1}^{\LL, \KK}"]&& \,\,\,\,\,\,\,\,\,\, & \\
        C_{q+1}^\LL\arrow{rr}{\partial_{q+1}^\LL}  && C_q^\LL\arrow{rr}{\partial_q^\LL} && C_{q-1}^\LL
\end{tikzcd}
\caption{The relevant morphisms in the definition of the persistent Laplacian $\triangle_{q}^{\KK, \LL}$.}
\label{fig_dig_per_lap}
\end{figure}

The commutative diagram associated with $\triangle_{q}^{\KK, \LL}$ is displayed in \cref{fig_dig_per_lap}. In this paper, we shall restrict our attention to the up persistent Laplacian $\triangle_{q, \mathrm{up}}^{\KK, \LL}$; see \cref{rem:full-laplacian}.

We recall the matrix representation of $\triangle_{q, \mathrm{up}}^{\KK, \LL}$ in the following theorem.

\begin{theorem}[{\cite[Theorem 3.1]{pers_lap}}]\label{thm_memoli}
Assume $n_{q+1}^{\LL, \KK}\coloneqq \dim(C_{q+1}^{\LL, \KK})>0$.
Fix any basis of $C_{q+1}^{\LL, \KK}\subset C_{q+1}^{\LL}$, represented as a matrix $Z\in \mathbb{R}^{n_{q+1}^\LL\times n_{q+1}^{\LL, \KK}}$, and let $B_{q+1}^{\LL, \KK}$ be the corresponding matrix representation of $\partial_{q+1}^{\LL, \KK}$.
Let $W_{q}^\KK$ and $W_q^\LL$ denote the diagonal weight matrix representations of $w_{q}^\KK$ and $w_q^\LL$, respectively. The matrix representation $\Delta_{q, \mathrm{up}}^{\KK, \LL}$ of $\triangle_{q, \mathrm{up}}^{\KK, \LL}$ can be expressed as:
\begin{equation}\label{eqn:per_lap}
\Delta_{q, \mathrm{up}}^{\KK, \LL}
=
B_{q+1}^{\LL, \KK} W_{q}^{\LL, \KK}\left(B_{q+1}^{\LL, \KK}\right)^{T}(W_q^\KK)^{-1},
\end{equation}
where $W_{q}^{\LL, \KK}:=\left(Z^T (W_{q+1}^\LL)^{-1} Z\right)^{-1}$.
\end{theorem}

We need to recall one more lemma. First, sort all the $q$-simplices $\bar{S}_q^\LL$,
and denote this ordered set by $[n_q^\LL]$. Let $I_q^\KK$ denote the indices of the $q$-simplices $S_q^\KK$ in $\LL$, and let $I_q^{\LL, \KK}:=[n_{q}^\LL]\backslash I_q^\KK$ denote the indices of $q$-simplices in $\LL$ but not in $\KK$. Lastly, let $D_{q+1}^{\LL, \KK}\coloneqq B_{q+1}^\LL(I_q^{\LL, \KK}, :)$.
As is evident from \eqref{eqn:per_lap}, the key point of computing $\Delta_{q, \mathrm{up}}^{\KK, \LL}$ is to find expressions for $B_{q+1}^{\LL, \KK}$ and $Z$, and this can be done via \cref{thm:weakfactor}.
The following lemma is a corollary of \cite[Lemma 3.4]{pers_lap}.

\begin{lemma}\label{lemma:pers_lap_memoli}
Let $\KK\hookrightarrow \LL$ be a pair of  simplicial complexes where $\LL$ is $q$-non-branching, and let $R_{q+1}^{\LL}:= D_{q+1}^{\LL, \KK} Y$ be a weak column reduction of $D_{q+1}^{\LL}$.
Moreover, let $I\subset [n_{q+1}^\LL]$ be the index set of $0$ columns of $R_{q+1}^\LL$.
Then, the following hold:
\begin{enumerate}
    \item If $I=\emptyset$,
    then $C_{q+1}^{\LL, \KK}=\{0\}$.
    \item If $I\neq \emptyset$ and $Z\coloneqq Y(:, I)$,
    then the columns of $Z$ constitute a basis for $C_{q+1}^{\LL, \KK}$.
    Moreover,
    $B_{q+1}^{\LL, \KK}\coloneqq (B_{q+1}^\LL Y)(I_q^{\KK}, I)$ is the matrix representation of $\partial_{q+1}^{\LL, \KK}$.
\end{enumerate}
\end{lemma}

\subsection{Computation of the up persistent Laplacian}\label{subsec:regular_form}
Let $\KK\hookrightarrow \LL$ be an inclusion of simplicial complexes for which $\LL$ is $q$-non-branching.
In this section, 
we compute the matrix representation of the up persistent Laplacian $\Delta_{q, \textrm{up}}^{\KK, \LL}$.
We denote by $C_0=(c_0^1, c_0^2, \cdots, c_0^h)$ the indices of zero columns in $D_{q+1}^{\LL, \KK}$,
and decompose the column indices $C$ of $D_{q+1}^{\LL, \KK}$ as
\begin{equation}
\label{eq:decomp}
C = C_0 \bigcup
\left(\bigcup_{i=1}^r C_i\right) \bigcup 
\left(\bigcup_{i=1}^q \bar{C}_i\right)\bigcup \left(\bigcup_{i=1}^s \widetilde{C}_i\right).
\end{equation}
Following the notation at the beginning of \cref{subsec:non_branching_mtx}, noting that $D_{q+1}^{\LL}$ is non-branching. According to \cref{thm:weakfactor},
there is an upper-triangular matrix $V$ with entries in $\{-1, 0, 1\}$ such that $R=D_{q+1}^{\LL, \KK} V$ is a weak column reduction.
Moreover,
the $\{c_{1}^{l_1}, \cdots, c_r^{l_r}\}$ columns in $V$ form an orthogonal basis (with respect to the standard inner product) for the kernel of the matrix $D_{q+1}^{\LL, \KK}$ with the 0-columns removed.

\begin{proposition}\label{prop:pers_lap}
Let $R=D_{q+1}^{\LL,\KK}V$ be a weak column reduction.
Then the following hold:
\begin{enumerate}
    \item The matrix $R$ contains $(h+r)$ zero columns,
    and the corresponding column indices are 
    \[\{c_0^1, \cdots, c_0^h, c_1^{l_1},\cdots, c_r^{l_r}\}.\]
    \item The submatrix $Z:=V(:, \{c_0^1, \cdots, c_0^h, c_1^{l_1},\cdots, c_r^{l_r}\})$ constitutes a basis of $C_{q+1}^{\LL, \KK}$.
    \item For 
    $b_j:= B_{q+1}^\LL(I_q^\KK, C_j)V(C_j, c_j^{l_j})$ \[B_{q+1}^{\LL, \KK}
    \coloneqq 
    B_{q+1}^\LL(I_q^{\KK}, C_0) \#
    \begin{bmatrix}
    b_1 & b_2 & \cdots & b_r
    \end{bmatrix}
    \]
    is the matrix representation of $\partial_{q+1}^{\LL, \KK}$ with respect to the basis $Z$.
\end{enumerate}
\end{proposition}
\begin{proof}
    The first statement follows immediately from  \cref{thm:main_thm} if we fix the zero columns of $D_{q+1}^{\LL, \KK}$ and compute a weak column reduction from its non-zero columns. The second statement is then a corollary of  \cref{lemma:pers_lap_memoli}.
    The third statement follows from \cref{lemma:pers_lap_memoli} as well.
    In fact, according to \cref{lemma:pers_lap_memoli} we have 
    $B_{q+1}^{\LL, \KK}=B_{q+1}^\LL(I_q^\KK, :)V\left(:, \{c_0^1,\cdots, c_0^{h}, c_1^{l_1}, \cdots, c_r^{l_r}\}\right)$.
    Recall that $\{c_0^1,\cdots, c_0^h\}$ are the indices of zero columns in $D_{q+1}^{\LL, \KK}$ and therefore for each $1\leq i \leq h$ the column vector $V\left(:, c_0^i\right)$ contains only one $1$ at row $c_0^i$ and all the remaining entries are $0$'s.
    Therefore,
    $B_{q+1}^\LL(I_q^\KK, :)V\left(:, C_0\right)=B_{q+1}^\LL(I_q^{\KK}, C_0)$.
    It thus remains to prove that $b_j=B_{q+1}^\LL(I_q^\KK, :)V(:, c_j^{l_j})$ for each $1\leq j\leq r$.
    In fact,
    according to  \cref{thm:main_thm}, the elements of $V(C_j, c_j^{l_j})$ are in $\{-1, 1\}$, and all remaining entries in the $c_j^{l_j}$-row are $0$'s. Therefore,  $B_{q+1}^\LL(I_q^\KK, :)V(:, c_j^{l_j})=B_{q+1}^\LL(I_q^\KK, C_j)V(C_j, c_j^{l_j})$ and the third statement is proved.
\end{proof}
We arrive at our main result in this section, where we follow the notation from the previous proposition, let $s_j:=\sum\limits_{k\in C_j}\frac{1}{W_{q+1}^\LL(k, k)}$ for $1\leq j\leq r$, and $S=\diag(s_1, s_2,\cdots, s_r)$ be a diagonal matrix.
\begin{theorem}
    \label{thm:fastcomputation}
    Let $\KK \hookrightarrow \LL$ be simplicial complexes, where $\LL$ is $q$-non-branching. The up persistent Laplacian $\Delta_{q, \mathrm{up}}^{\KK, \LL}$ is given by 
    \begin{equation}\label{eq:uplap_1}
    \Delta_{q, \mathrm{up}}^{\KK, \LL}
    =
    B_{q+1}^{\LL, \KK} W_{q+1}^{\LL, \KK}\left(B_{q+1}^{\LL, \KK}\right)^{T}(W_{q}^\KK)^{-1},
    \end{equation}
    where $W_{q+1}^{\LL, \KK}$ equals the diagonal matrix $W_{q+1}^{\LL}(C_0, C_0) \oplus S^{-1}$.
\end{theorem}
\begin{proof}
    Since the columns of $V(:, \{c_1^{l_1}\cdots, c_r^{l_r}\})$ are orthogonal to each other (with respect to the standard inner product), the statement is immediate from Equation \eqref{eqn:per_lap} and \cref{prop:pers_lap}.

\end{proof}

\begin{example}\label{ex_2}
Consider the pair of simplicial complexes $\KK\hookrightarrow \LL$ in \cref{fig_mobius}. Here, the simplicial complex $\LL$ is a triangulation of the of M\"obius strip, and $\KK$ is its boundary. The matrix representations of $\partial_2^{\LL}$ and $D_2^{\LL, \KK}$ can be expressed as:
\[
B_2^{\LL}
=
\begin{bNiceMatrix}[
  first-row,code-for-first-row=\scriptstyle,
  first-col,code-for-first-col=\scriptstyle,
]
& [134] & [124] & [245] & [235] & [135]\\
[34] & 1 & 0 & 0 & 0 &0 \\
[14] & -1 & -1 & 0 & 0 & 0 \\
[13] & 1 & 0 & 0 & 0 & 1\\
[24] & 0 & 1 & 1 & 0 & 0 \\
[12] & 0 & 1 & 0 & 0 & 0\\
[45] & 0 & 0 & 1 & 0 & 0\\
[25] & 0 & 0 & -1 & -1 & 0 \\
[35] & 0 & 0 & 0 & 1 & 1 \\
[23] & 0 & 0 & 0 & 1 & 0\\
[15] & 0 & 0 & 0 & 0 & -1\\
\end{bNiceMatrix},\qquad\qquad 
D_2^{\LL, \KK}
=
\begin{bNiceMatrix}[
  first-row,code-for-first-row=\scriptstyle,
  first-col,code-for-first-col=\scriptstyle,
]
& [134] & [124] & [245] & [235] & [135]\\
[14] & -1 & -1 & 0 & 0 & 0 \\
[13] & 1 & 0 & 0 & 0 & 1 \\
[24] & 0 & 1 & 1 & 0 & 0 \\
[25] & 0 & 0 & -1 & -1 & 0 \\
[35] & 0 & 0 & 0 & 1 & 1 \\
\end{bNiceMatrix}.
\]
\begin{figure}[h]
  \begin{subfigure}[b]{0.4\textwidth}
    \begin{tikzpicture}
    \def\side{2.251666}


    \filldraw[fill=yellow!40] (0, 0) -- (\side, 0) -- (1/2 * \side, \side * 0.8660254) -- cycle;
    \filldraw[fill=yellow!40] (\side, 0) -- (1/2 * \side, \side * 0.8660254) -- (\side * 3/2, \side * 1.732051 / 2) -- cycle;
    \filldraw[fill=yellow!40] (\side, 0) -- (2 * \side, 0) -- (\side * 3/2, \side * 1.732051 / 2) -- cycle;
    \filldraw[fill=yellow!40] (\side * 2, 0) -- (5/2 * \side, \side * 0.8660254) -- (\side * 3/2, \side * 1.732051 / 2) -- cycle;
    \filldraw[fill=yellow!40] (5/2 * \side,  \side * 0.8660254) -- (2 * \side, -0) -- (3 * \side, -0) -- cycle;



    \node at (1/2 * \side,  \side * 0.8 + 0.4) {1};
    \node at (3/2 * \side,  \side * 0.8 + 0.4) {2};
    \node at (5/2 * \side,  \side * 0.8 + 0.4) {3};
    \node at (0, -0.3) {3};
    \node at (\side, -0.3) {4};
    \node at (2 * \side, -0.3) {5};
    \node at (3 * \side, -0.3) {1};

    \fill (0, 0) circle (2pt);  
    \fill (1/2 * \side, \side * 0.8660254) circle (2pt); 
    \fill (3/2 * \side, \side * 0.8660254) circle (2pt);
    \fill (5/2 * \side, \side * 0.8660254) circle (2pt);
    \fill (\side, 0) circle (2pt);
    \fill (2 * \side, 0) circle (2pt);
    \fill (3 * \side, -0) circle (2pt);
\end{tikzpicture}
\caption{$\LL$.}%
\end{subfigure}
\qquad\qquad
\begin{subfigure}[b]{0.4\textwidth}
    \begin{tikzpicture}
    \def\side{2.251666}


    
    \draw (\side, 0) -- (2 * \side, 0);
    
    \draw (0, 0) -- (\side, 0);
    \draw (1/2 * \side, \side * 0.8660254) -- (3/2 * \side, \side * 0.8660254);
    
    \draw (\side * 3/2, \side * 1.732051 / 2) -- (5/2 * \side, \side * 0.8660254);
    \draw (2 * \side, -0) -- (3 * \side, -0);



    \node at (1/2 * \side,  \side * 0.8 + 0.4) {1};
    \node at (3/2 * \side,  \side * 0.8 + 0.4) {2};
    \node at (5/2 * \side,  \side * 0.8 + 0.4) {3};
    \node at (0, -0.3) {3};
    \node at (\side, -0.3) {4};
    \node at (2 * \side, -0.3) {5};
    \node at (3 * \side, -0.3) {1};

    \fill (0, 0) circle (2pt);  
    \fill (1/2 * \side, \side * 0.8660254) circle (2pt); 
    \fill (3/2 * \side, \side * 0.8660254) circle (2pt);
    \fill (5/2 * \side, \side * 0.8660254) circle (2pt);
    \fill (\side, 0) circle (2pt);
    \fill (2 * \side, 0) circle (2pt);
    \fill (3 * \side, -0) circle (2pt);
\end{tikzpicture}
\caption{$\KK$.}%
  \end{subfigure}
\caption{Simplicial complexes $\KK\hookrightarrow \LL$ associated with Example \ref{ex_2}.}
\label{fig_mobius}
\end{figure}
It is easy to check that $D_2^{\LL, \KK}$ is irregular (\cref{def:reg_mtrx}), and that there is only one component in the undirected graph associated to $|D_2^{\LL, \KK}|^T$. Hence $D_2^{\LL, \KK}$ has full column rank and thus $\Delta_{1, \mathrm{up}}^{\KK, \LL}=0$.
\end{example}

In the following proposition we show that the up persistent Laplacian $\Delta_{q, \mathrm{up}}^{\KK, \LL}$ can be computed in quadratic arithmetic operations,
given that the input data are in the format of sparse (CSR) matrix.
We exhibit the pseudocode of computing up persistent Laplacian in \cref{alg:pers_lap}.

\begin{proposition}
\label{prop:alg_lap}
Let $\KK\hookrightarrow \LL$ be a pair of simplicial complexes as described in \cref{thm:fastcomputation}.
The up persistent Laplacian $\Delta_{q, \text{up}}^{\KK, \LL}$ can be computed in $O\left((n_q^\LL-n_q^\KK)\times \alpha(n_{q+1}^\LL) + n_{q+1}^\LL+(n_q^\KK)^2\right)$.
\end{proposition}

\begin{proof}
Computing the weak column reduction of the matrix $D_{q+1}^{\LL, \KK}$ takes $O\left((n_q^\LL-n_q^\KK)\times \alpha(n_{q+1}^\LL) + n_{q+1}^\LL\right)$ according to \cref{prop:alg_weak_red}.
Since the matrix $B_{q+1}^\LL$ is non-branching and $W_{q+1}^{\LL, \KK}$ is diagonal,
there are at most $2n_{q}^\KK$ non-zero entries in $B_{q+1}^{\LL,\KK}$ and $n_{q}^\KK$ non-zero entries in $W_{q+1}^{\LL, \KK}$, respectively,
hence the computation of $B_{q+1}^{\LL, \KK}$ and $W_{q+1}^{\LL, \KK}$ takes $O(n_q^\KK)$.
Finally,
computing the matrix product of $\Delta_{q, \text{up}}^{\KK, \LL}$ takes $O((n_q^\KK)^2)$.
To sum up,
it takes $O\left((n_q^\LL-n_q^\KK)\times \alpha(n_{q+1}^\LL) + n_{q+1}^\LL+(n_q^\KK)^2\right)$ to compute $\Delta_{q, \text{up}}^{\KK, \LL}$.
\end{proof}

\begin{algorithm}[H]
 \caption{\textsc{UpPersistentLaplacian}}
 \label{alg:pers_lap}
 \begin{algorithmic}[1]
 \State \textbf{Input:} $I_q^\KK$, $B_{q+1}^\LL$, $W_{q+1}^\LL$, $W_{q}^\KK$
 \State \textbf{Output:} $\Delta_{q, \mathrm{up}}^{\KK, \LL}$
\State $\Delta_{q, \mathrm{up}}^{\KK, \LL}\gets 0_{n_q^\KK\times n_q^\KK}$ 
\Comment{Initialize $\Delta_{q, \mathrm{up}}^{\KK, \LL}$}
 \If{$n_q^\KK=n_q^\LL$}
 \State $\Delta_{q, \mathrm{up}}^{\KK, \LL}\gets B_q^\LL \left(B_q^\LL\right)^T$
 \Else
 \State $D_{q+1}^{\LL, \KK}\gets B_{q+1}^\LL(I_q^{\LL, \KK}, :)$

 \State $C_0\gets \text{ indices of zero columns in }D_{q+1}^{\LL, \KK}$
 \State $C_{\text{nonzero}}\gets \text{indcies of non-zero columns in }D_{q+1}^\LL$
 

\State{$R, E, V\gets \textsc{WeakColumnReduction}(D_{q+1}^\LL(:, C_{\text{nonzero}}))$}

\Statex

\For{$j$ in the zero column indices of $R$}

\State $C_j\gets \text{support indices of the column vector }V(:, j)$
\State $s_j\gets \sum_{i\in C_j}\frac{1}{W_{q+1}^{\LL}(i, i)}$
\State $b_j\gets \text{the sum of column vectors in } B_{q+1}^\LL(I_q^\KK, I_j)E(C_j, C_j)$
\EndFor

\State $B_{q+1}^{\LL, \KK}\gets
    B_{q+1}^\LL(I_q^{\KK}, C_0) \#
    \begin{bmatrix}
    b_1 & b_2 & \cdots & b_r
    \end{bmatrix}$


\State $W_{q+1}^{\LL, \KK}\gets W_{q+1}^\LL(C_0, C_0)\oplus S^{-1}$

\State $\Delta_{q, \text{up}}^{\KK, \LL}\gets B_{q+1}^{\LL, \KK}W_{q+1}^{\LL, \KK}\left(B_{q+1}^{\LL, \KK}\right)^T (W_q^\KK)^{-1}$

\EndIf
\end{algorithmic}
\end{algorithm}

\begin{remark}\label{rmk:cmpt_time}
In \cref{prop:alg_lap},
we notice that the quadratic arithmetic operations of $\Delta_{q, \text{up}}^{\KK, \LL}$ is due to the matrix product operation, 
which is reasonable since even if all the matrices $B_{q+1}^{\LL, \KK}$, $W_{q+1}^{\LL, \KK}$ and $W_q^\KK$ are in the sparse format, 
we have to scan over every entry of $\Delta_{q, \text{up}}^{\KK, \LL}$.
Therefore if we only compute the pair of matrices $(B_{q+1}^{\LL, \KK}, W_{q+1}^{\LL, \KK})$ instead of the up persistent Laplacian $\Delta_{q, \text{up}}^{\KK, \LL}$ explicitly,
it takes $O\left((n_q^\LL-n_q^\KK)\times \alpha(n_{q+1}^\LL) + n_{q+1}^\LL+n_q^\KK\right)$ arithmetic operations.
\end{remark}

\subsection{The up persistent Laplacian for filtrations}\label{sec:filtration}

One advantage of DSU algorithm is that it 
supports near-constant amortized time complexity 
in handling dynamic graphs that only undergo edge additions.
As a result,
we can compute the weak column reduction with respect to a sequence of non-branching matrices efficiently.

\begin{lemma}\label{lemma:D_filtration}
Let $\{D_0, D_1, \cdots, D_m\}$ be a sequence of non-branching matrices where each $D_{i}$ contains exactly one more row than $D_{i-1}$ for each $1\leq i\leq m$.
Suppose $D_0$ is a $k\times l$ matrix,
then we can find all the weak column reduction of all $\{D_0, D_1, \cdots, D_m\}$ in
$$O\left(m\times (k+m+l) + (k+m)\times \alpha(l)\right)$$ arithmetic operations.
\end{lemma}

\begin{proof}
Without loss of generality,
we assume that in the matrix $D_{i}$ the last row is the one more extra row than $D_{i-1}$ for each $1\leq i\leq m$.
We denote by $G_i$ the undirected graph whose incidence matrix is $|D_i|^T$.
According to \cref{prop:alg_weak_red} it takes $O(k\times \alpha(l)+l)$ to compute the weak column reduction with respect to $D_0$.

To compute that of $D_1$, 
we need to first find the connected components of $G_1$,
since $G_1$ contains exactly one more edge or loop than $G_0$,
either two components in $G_0$ are getting connected or one component in $G_0$ is added one more self-loop,
hence we only need to apply \textsc{Union} and \textsc{Find} actions once hence the arithmetic operation is $\alpha(l)$.
According to the proof of \cref{thm:weakfactor} in \cref{subsec:non_branching_mtx},
to compute the weak column reduction of $D_1$ we only need to check whether the new component in $G_1$ is regular or not.
In more detail,
suppose $C_1, C_2, \cdots, C_r$ are the regulable components in $D_0$,
i.e.,
$D_0(:, C_i)$ is a regulable matrix for each $1\leq i\leq r$,
and the weak column reduction is $R_0=D_0 E_0 V_0$.
In the case that the row vector $D_1(-1, :)$ contains one non-zero entry,
whose column index is $j_1$,
we then only need to check if $j_1$ is in any of the regulable components,
if so,
say,
$j_1\in C_i$
,
then $R_1$ is given by replacing $R_0(:, j_1)$ by $D_0(:, j_1)$ and 
adding one more additional row vector $D_1(-1, :)$ in the bottom.
If $j_1$ does not belong to any of $C_i$ for $1\leq i\leq r$,
it means that the regulable components are not affected by this additional loop,
hence we can get $R_1$ simply be appending the row vector $D_1(-1, :)$ in the bottom of the matrix $R_0$.
Similarly,
in the case that the row vector $D_1(-1, :)$ contains two non-zero entries,
that $D_1(-1, j_1)=1$ and $D_1(-1, j_2)=-1$,
with the corresponding column indices being $j_1 < j_2$,
if $j_1\in C_{i_1}$ and $j_2\in C_{i_2}$ for some $1\leq i_1, i_2\leq r$,
then the \textsc{Union} action unite the two sets $C_{i_1}$ and $C_{i_2}$ into one set,
and to compute $R_1$,
we only need to replace $R_0(:, j_1)$ by $D_{0}(:, j_1)$,
and append one more additional row vector 
$\begin{bNiceMatrix}[
  first-row,code-for-first-row=\scriptstyle,
]
   &            & j_1 &  \\
0 & \cdots & 1   & 0  \cdots & 0\\
\end{bNiceMatrix},$
in the bottom.
In the other cases,
the discussions are similar.
Notice that in each case to calculate $R_1$ we only apply column action to one or two columns in $R_0$ and append an extra row in the bottom of $R_0$,
thus it costs $O(k+l)$ arithmetic operations.
In a similar manner,
we can calculate $E_1$ and $V_1$ out of $E_0$ and $V_0$ in $O(l)$.
Hence the weak column reduction of $D_1$ can be computed in $O(k+l)$.
Similarly,
we can compute the weak column reduction of $D_i$ from that of $D_{i-1}$ for $1\leq i\leq m$.
Thus to sum up it takes $O\left(m\times (k+m+l) + (k+m)\times \alpha(l)\right)$.
\end{proof}

\begin{remark}\label{rmk:filt_weak_red}
In some situations,
we are only interested in the null space of $R_i$ and it is not necessary to write out $R_i$ explicitly.
On the other hand, 
in the proof of \cref{lemma:D_filtration} we observe that when computing $R_i$,
only by applying \textsc{Union} and \textsc{Find} once we can get the indices of zero columns in $R_i$,
which takes $O(\alpha(l))$.
Therefore it takes $O\left((k+m)\times \alpha(l)+l\right)$ to find the dimension of null spaces of each matrix in $\{D_0, \cdots, D_m\}$. 
\end{remark}

Let $\KK_0\hookrightarrow\KK_1 \hookrightarrow\cdots\hookrightarrow \KK_m\hookrightarrow \LL$ be a $q$-non-branching simplicial filtration where each $\KK_i$ contains exactly one more $q$-simplex than $\KK_{i-1}$ for each $1\leq i\leq m$.
Without loss of generality,
we assume that the set of $q$-simplices $S_q^{\KK_i}$ in $\KK_i$ is ordered such that the additional $q$-simplex of $\KK_i$ is at the end of $S_q^{\KK_i}$,
and $S_q^{\KK_m}=\{\sigma_i|\sigma_i\in S_q^{\LL}\}_{i=1}^{n_q^{\KK_m}}$.

\begin{proposition}\label{prop:time_filtration}
The sequence of up persistent Laplacians $\{\Delta_{q, \mathrm{up}}^{\KK_i, \LL}\}_{i=0}^m$ can be computed in 
\[
O\left((m+1)\times (n_q^{\KK_m})^2+(m+1)\times n_{q+1}^\LL+(n_q^\LL-n_q^{\KK_m}+m)\times \alpha(n_{q+1}^\LL)\right)
\]
arithmetic operations.
\end{proposition}

\begin{proof}
According to \cref{prop:alg_lap},
it takes $O\left((n_q^\LL-n_q^{\KK_m})\times \alpha(n_{q+1}^\LL) + n_{q+1}^\LL+(n_q^{\KK_m})^2\right)$ to compute $\Delta_q^{\KK_m, \LL}$.
Let $D_{q+1}^{\LL, \KK_i}:=B_{q+1}^\LL(I_{q}^{\LL, \KK_i}, :)$ and $R^{i}=D_{q+1}^{\LL, \KK_i}E^{i}V^{i}$ be the weak column reduction of $D_{q+1}^{\LL, \KK_i}$.
We then evaluate the arithmetic operations of computing $\Delta_q^{\KK_{m-1}, \LL}$.
First we need to find the indices of zero columns in $R^{m-1}$,
according to \cref{rmk:filt_weak_red} it costs nearly constant arithmetic operation $O(\alpha(n_{q+1}^\LL))$,
besides that,
we also need to compute $E^{m-1}$ and $V^{m-1}$,
which costs $O(n_{q+1}^\LL)$ as explained in the proof of \cref{lemma:D_filtration}.
We then need to compute $B_{q+1}^{\LL, \KK_{m-1}}$.
Recall that the columns in $B_{q+1}^{\LL, \KK_i}$ are one-to-one corresponding to the regulable components in $D_{q+1}^{\LL, \KK_i}$,
thus we can derive $B_{q+1}^{\LL, \KK_{m-1}}$ from $B_{q+1}^{\LL, \KK_m}$ by modifying one or two columns and removing one row,
which costs $O(n_{q}^{\KK_m})$.
Computing $W_{q+1}^{\LL, \KK_{m-1}}$ from $W_{q+1}^{\LL, \KK_{m}}$ costs $O(1)$,
therefore it takes $O(\alpha(n_{q+1}^\LL)+n_{q+1}^\LL+n_q^{\KK_m}+(n_{q}^{\KK_{m-1}})^2)$ to compute $\Delta_{q, \mathrm{up}}^{\KK_{m-1}, \LL}$,
where the quadratic term comes from computing the matrix product.
By iteration,
we obtain the arithmetic operations of $\{\Delta_{q, \mathrm{up}}^{\KK_i, \LL}\}_{i=0}^m$.
\end{proof}

In \cref{prop:time_filtration} we showed that our algorithm requires cubic arithmetic operations to compute the set $\{\Delta_{q, \mathrm{up}}^{\KK_i, \LL}\}_{i=0}^m$. This cubic running time results from the need to perform matrix product operations, and although the matrices $B_{q+1}^{\LL, \KK_i}$, $W_{q+1}^{\LL, \KK}$ and $W_q^{\KK_i}$ are non-branching, 
it is not guaranteed the matrix $\Delta_{q, \mathrm{up}}^{\KK_i, \LL}$ is sparse. 

\begin{example}\label{ex:dense}
  Consider the pair of simplicial complexes $\KK\hookrightarrow \LL$ in \cref{fig:dense_example}. By a direct calculation we obtain that 
$B_{2}^{\LL, \KK}
=
\begin{bmatrix}
    1 & -1 & 1 & -1
\end{bmatrix}^T$,
and hence $\Delta_{1, \text{up}}^{\KK, \LL}$ is dense. In particular, obtaining each of the up persistent Laplacians in the filtration requires a quadratic number of arithmetic operations. This example generalizes to more edges (larger matrices) by triangulating an $n$-gon and letting $\KK$ be the outer boundary.
\begin{figure}[H]
  \centering
  \begin{subfigure}[b]{0.2\textwidth}
\begin{tikzpicture}
    \def\side{1.}

    
   
    \filldraw[fill=white] (-\side, -\side) -- (\side, -\side) -- (\side, \side) -- (-\side, \side) -- (-\side, -\side) -- cycle;

    
    \fill (-\side, \side) circle (2pt);  
    \fill (-\side, -\side) circle (2pt); 
    \fill (\side, \side) circle (2pt); 
    \fill (\side, -\side) circle (2pt);  

    \node at (-\side, \side) [above left] {1};
    \node at (-\side, -\side) [below left] {2};
    \node at (\side, \side) [above right] {3};
    \node at (\side, -\side) [below right] {4};
\end{tikzpicture}
\caption{$\KK$.}
\end{subfigure}%
\quad
 \begin{subfigure}[b]{0.2\textwidth}
\begin{tikzpicture}
    \def\side{1.}

    
   
    \filldraw[fill=yellow!40] (-\side, -\side) -- (\side, -\side) -- (\side, \side) -- (-\side, \side) -- (-\side, -\side) -- cycle;
    \draw (-\side, -\side) -- (\side, \side);

    
    \fill (-\side, \side) circle (2pt);  
    \fill (-\side, -\side) circle (2pt); 
    \fill (\side, \side) circle (2pt); 
    \fill (\side, -\side) circle (2pt);  

    \node at (-\side, \side) [above left] {1};
    \node at (-\side, -\side) [below left] {2};
    \node at (\side, \side) [above right] {3};
    \node at (\side, -\side) [below right] {4};
\end{tikzpicture}
\caption{$\LL$.}
\end{subfigure}%
\caption{The matrix $\Delta_{1, \text{up}}^{\KK, \LL}$ associated to these two complexes is dense.}
\label{fig:dense_example}
\end{figure}

\end{example}

However,
if we do not need to compute the up persistent Laplacians explicitly (e.g., if we only want the eigenvalues; see \cref{sec:eigenvalues}), the arithmetic operations can be reduced to be quadratic. We summarize the result in the following corollary:

\begin{corollary}\label{cor:filt}
The sequences of matrix pairings $\left\{\big(B_{q+1}^{\LL, \KK_i}, W_{q+1}^{\LL, \KK_i}\big)\right\}_{i=0}^m$ and the matrix products 
\[\left\{(W_q^{\KK_i})^{-1/2}B_{q+1}^{\LL, \KK_i} (W_{q+1}^{\LL, \KK_i})^{1/2}\right\}_{i=0}^m\] 
can be computed in 
\[
O\left((n_q^\LL-n_q^{\KK_m}+m)\times \alpha(n_{q+1}^\LL) + (m+1)\times (n_{q+1}^\LL + n_q^{\KK_m})\right)
\]
arithmetic operations.
\end{corollary}

\section{A generalized Kron reduction and a Cheeger inequality}\label{subsec:kron_red}

In graph theory, the eigenvalues of the graph Laplacian (specifically, the up Laplacian) encode crucial information about connectivity, clustering, and expansion. For instance, the smallest non-zero eigenvalue reflects how well-connected the graph is, as captured by the Cheeger inequalities. Recently, a Cheeger-type inequality for the Hodge Laplacian on simplicial complexes was established in~\cite{jost2023cheeger}, and a Cheeger inequality for the persistent graph Laplacian was given in~\cite{pers_lap}. However, as noted in~\cite{wei2023}, very little is known about the topological interpretation of the eigenvalues and eigenvectors of persistent Laplacians. In this section, we show that the setting of non-branching complexes provides a clearer understanding of the eigenvalues of the up persistent Laplacian.

Concretely, in \cref{sec:kron}, we show that $\Delta_{q, \mathrm{up}}^{\KK, \LL}$ coincides with the Laplacian of a weighted oriented hypergraph, under the assumption that $\LL$ is unweighted and $q$-non-branching. The relevant theory for hypergraphs is provided in \cref{appdx_hg}. This type of reduction—where the persistent (graph) Laplacian is replaced by the Laplacian of a single graph—is known as a \emph{Kron reduction}. However, as shown in~\cite[Section SM5]{pers_lap}, one cannot, in general, replace the up persistent Laplacian of a pair of simplicial complexes with the Laplacian of a single simplicial complex. Exploring generalizations of Kron reduction to the simplicial setting is therefore left by the authors as an open direction for future work. Our result thus provides a partial answer to this open problem.

Leveraging the connection to hypergraphs, we establish a Cheeger-type inequality for the smallest non-zero eigenvalue in \cref{sec:cheeger}. To the best of our knowledge, this is the first result of its kind for the up persistent Laplacian of simplicial complexes.

\begin{remark}
While we do not make the statements explicit, all of the following results apply verbatim to the case of cubical complexes.
\end{remark}

\subsection{The Kron reduction}
\label{sec:kron}
Let $\LL$ be an unweighted $q$-non-branching simplicial complex. 
Recall that $C_j$ is a column-index subset of $B_{q+1}^\LL$ associated with the vertices of a connected component in a hypergraph, 
and $l_j = |C_j|$ is the cardinality of $C_j$. 
Note that $B_{q+1}^\LL(:, C_j)$ corresponds to $l_j$ distinct $(q+1)$-simplices,
and \cref{prop:pers_lap} implies that the column vector $B_{q+1}^{\LL, \KK}(:, j)$ is equal to the column sum of $B_{q+1}^\LL(:, C_j)$. Therefore, the column $B_{q+1}^{\LL, \KK}(:, j)$ can be interpreted as the boundary of a $(q+1)$-dimensional polyhedron formed by the union of the $l_j$ $(q+1)$-simplices.

In fact, we can interpret $B_{q+1}^{\LL, \KK}$ as the incidence matrix of an oriented  hypergraph $\widetilde{\KK}=(V, E)$, where each vertex corresponds to a $q$-simplex in $\KK$, and where the hyperedges are given by the boundaries of connected polyhedra in $\LL$.  Furthermore, since $\LL$ is unweighted, we recall from \cref{thm:fastcomputation} that
the up persistent Laplacian can be written as
\[
\Delta_{q, \mathrm{up}}^{\KK, \LL}
=
B_{q+1}^{\LL, \KK} \big(\id_h\oplus S^{-1}\big)\left(B_{q+1}^{\LL, \KK}\right)^{T},
\]
where $S^{-1}=\diag(l_1^{-1}, \cdots, l_r^{-1})$.
Therefore, if we regard $(\id_h\oplus S^{-1})$ as the hyperedge weights of the oriented hypergraph $\widetilde{\KK}$, the up persistent Laplacian $\Delta_{q, \mathrm{up}}^{\KK, \LL}$ coincides with the hypergraph Laplacian (\cref{def:laphyper}) of $\widetilde{\KK}$.
Summarized, we have the following.
\begin{proposition}\label{prop:cheeger}
Let $\LL$ be an unweighted $q$-non-branching simplicial complex. Then, the up persistent Laplacian takes the form
 \[\Delta_{q, \mathrm{up}}^{\KK, \LL}=B_{q+1}^{\LL, \KK}\left(\id_h\oplus S^{-1}\right)(B_{q+1}^{\LL, \KK})^T\]
which coincides with the hypergraph Laplacian of the weighted oriented hypergraph $\widetilde{\KK}$.
\end{proposition}

On the other hand,
we can also consider the dual hypergraph (\cref{def_dual_hg}) $\widetilde{\KK}_{\textrm{dual}}$,
i.e.,
we regard each polyhedron in $\LL$ as a vertex, and each $q$-simplex as a hyperedge. That is, we take the transpose of the incidence matrix of $\widetilde{\KK}$, to obtain a hypergraph where the vertex weights are given by $\id_h\oplus S^{-1}$ and the hyperedges are unweighted (weights $\equiv$ 1). The following is immediate.

\begin{proposition}
Let $\Delta_{\widetilde{\KK}_{\textrm{\textrm{dual}}}}$ be the Laplacian associated with the dual the hypergraph $\widetilde{\KK}_{\textrm{\textrm{dual}}}$.
Then, $
\Delta_{\widetilde{\KK}_{\textrm{\textrm{dual}}}}
=
(\id_h\oplus S^{-1})(B_{q+1}^{\LL, \KK})^TB_{q+1}^{\LL, \KK}.$
\end{proposition}

The following simple observation will be important in the next section.

\begin{lemma}\label{lemma:same_eig_val}
$\Delta_{\widetilde{\KK}_{\textrm{\textrm{dual}}}}$ and $\Delta_{q, \mathrm{up}}^{\KK, \LL}$ have the same non-zero eigenvalues.
\end{lemma}

\begin{proof}
If $\lambda$ is a non-zero eigenvalue of $\Delta_{\widetilde{\KK}_{\textrm{\textrm{dual}}}}$,
then there is an eigenvector $v$ such that
$(\id_h\oplus S^{-1})(B_{q+1}^{\LL, \KK})^TB_{q+1}^{\LL, \KK}v=\lambda v$.
Multiplying by $B_{q+1}^{\LL, \KK}$ on both sides,
we obtain that $\lambda$ is also an eigenvalue of $\Delta_{q,\textrm{up}}^{\KK, \LL}$ with the corresponding eigenvector being 
$B_{q+1}^{\LL, \KK}v$.
Similarly,
if $\lambda'$ is an eigenvalue of $\Delta_{q,\textrm{up}}^{\KK, \LL}$,
there is some eigenvector $v'$ such that
$B_{q+1}^{\LL, \KK}\left(\id_h\oplus S^{-1}\right)(B_{q+1}^{\LL, \KK})^T v'=\lambda' v'$.
Multiplying by $(\id_h\oplus S^{-1})$
$\left(B_{q+1}^{\LL, \KK}\right)^T$ on both sides,
we prove the lemma.
\end{proof}

\subsection{A Cheeger-type inequality}\label{sec:cheeger}
We denote by $P_j$ the $(q+1)$-dimensional (connected) polyhedron in $\LL$ corresponding to the column vector $B_{q+1}^{\LL, \KK}(:, j)$ for $1\leq j\leq h+r$.
Let $\mathbf{V}(P_j)$ be the number of $(q+1)$-simplices in the polyhedron $P_j$.
Moreover,
we let $\mathbf{A}(P_j)$ be the number of $q$-simplices forming the boundary of $P_j$, i.e., the number of non-zero entries of the column vector $B_{q+1}^{\LL, \KK}(:, j)$,
and let $\mathbf{A}(P_i\bigcap P_j)$ be the number of $q$-simplices shared by the boundaries of polyhedra $P_{i}$ and $P_{j}$,
i.e.,
$\mathbf{A}(P_i\bigcap P_j)$ is equal to the absolute value of the (standard) dot product of the column vectors $B_{q+1}^{\LL, \KK}(:, i)$ and $B_{q+1}^{\LL, \KK}(:, j)$.
Finally,
we let $\widehat{\mathbf{A}}(P_j)$ be the number of $q$-simplices in the boundary of $P_j$ which are not part of the boundary of any other polyhedron,
i.e.,
$\widehat{\mathbf{A}}(P_j)
=
\mathbf{A}(P_j)
-
\sum\limits_{i\neq j}\mathbf{A}(P_i\bigcap P_j).
$
We first recall the Ger\u{s}gorin circle theorem below,
the readers can consult \cite{MR2978290, MR2093409} for more details.
\begin{theorem}[Geršgorin]\label{thm:gersgorin}
Let $A=(a_{ij})\in M_n(\mathbb{C})$. Let $R_i=\sum\limits_{j\neq i}|a_{ij}|$. Let $D(a_{ii}, R_i)\subset \mathbb{C}$ be the closed disc centered at $a_{ii}$ with radius $R_i$, which is referred to as a Geršgorin disc. Then every eigenvalue of $A$ lies within the union of all the Geršgorin discs.
\end{theorem}
Denote by $\lambda_{\mathrm{min}}^{\KK, \LL}$ the minimal non-zero eigenvalue of $\Delta_{q, \mathrm{up}}^{\KK, \LL}$.

\begin{proposition}\label{prop:realcheeger1}
Let $\KK\hookrightarrow \LL$ where $\KK$ and $\LL$ are unweighted and $\LL$ is $q$-non-branching. Then,
\[
\min\limits_{j\leq h+r}\frac{\widehat{\mathbf{A}}(P_j)}{\mathbf{V}(P_j)}\leq \lambda_{\mathrm{min}}^{\KK, \LL}.\]
\end{proposition}

\begin{proof}
We notice that the $(i, j)$-entry of $\Delta_{\widetilde{\KK}_{\textrm{\textrm{dual}}}}$ is equal to $\frac{\mathbf{A}(P_i\bigcap P_j)}{\mathbf{V}(P_i)}$ with $i\neq j$,
and the absolute value of diagonal entries of $\Delta_{\widetilde{\KK}_{\textrm{\textrm{dual}}}}$ are given by $\frac{\mathbf{A}(P_i)}{\mathbf{V}(P_i)}$ for each $i\leq h+r$.
Thus the result follows from \cref{thm:gersgorin}.
\end{proof}

\begin{lemma}\label{lemma:full_col_rnk}
Suppose $B_{q+1}^{\LL, \KK}$ is not equal to $0$.
If $B_{q+1}^\LL$ has full column rank,
then $B_{q+1}^{\LL, \KK}$ has full column rank.
\end{lemma}
\begin{proof}
    From \cref{prop:pers_lap}, we have that $B_{q+1}^{\LL, \KK}$ is obtained from $B_{q+1}^\LL$ by performing left-to-right column operations, and removing the rows corresponding to simplices in $\LL-\KK$. Specifically, we have the following description: 
    \[B_{q+1}^{\LL, \KK}
    \coloneqq 
    B_{q+1}^\LL(I_q^{\KK}, C_0) \#
    \begin{bmatrix}
    b_1 & b_2 & \cdots & b_r
    \end{bmatrix}
    \]
    where the column indices $c_0^i$ from \cref{prop:pers_lap} (i.e., not the column indices of the $b_i$'s) correspond to columns in $B_{q+1}^\LL$ that are identically 0 on the rows that get removed. Furthermore,  each vector $b_i$ is obtained by adding to a column $B_q^\LL(:,c_i^{l_i})$ a number of columns in $B_q^\LL$ with column indices $k$ for $1\leq k< c_i^{l_i}$. Let the resulting column vector be $d_{i}$. Since $B_{q+1}^\LL$ is assumed to be of full column rank, we have that the matrix $A$ obtained by performing these column operations on $B_{q+1}^\LL$ is of full column rank. Moreover, by construction, each vector $d_{i}$ is 0 at the rows that ultimately get removed. Hence, we obtain the matrix $B_{q+1}^\LL$ by omitting rows of zeros from $A$. The result follows.
\end{proof}

\begin{proposition}\label{prop:realcheeger2}
Let $\KK\hookrightarrow \LL$ where $\LL$ is unweighted and $q$-non-branching. 
Suppose the boundary matrix $B_{q+1}^\LL$ has full column rank and $B_{q+1}^{\LL, \KK}$ is not $0$.
Then,
$
\lambda_{\mathrm{min}}^{\KK, \LL}\leq \min\limits_{j\leq h+r} \frac{\mathbf{A}(P_j)}{\mathbf{V}(P_j)}.
$
\end{proposition}

\begin{proof} 
Since $B_{q+1}^{\LL}$ has full column rank, the matrix $B_{q+1}^{\LL, \KK}$ has full column rank as well; see \cref{lemma:full_col_rnk}.
Hence, the matrix $\Delta_{\widetilde{\KK}_{\textrm{\textrm{dual}}}}$ is invertible 
and therefore $\Delta_{\widetilde{\KK}_{\textrm{\textrm{dual}}}}$ is positive-definite. By the Courant-Fischer-Weyl Min-Max theorem for eigenvalues of positive-definite matrices, we conclude that the minimal eigenvalue of $\Delta_{\widetilde{\KK}_{\textrm{\textrm{dual}}}}$ is at most 
the smallest diagonal entry of $\Delta_{\widetilde{\KK}_{\textrm{\textrm{dual}}}}$ which equals $\min\limits_{j\leq h+r} \frac{\mathbf{A}(P_j)}{\mathbf{V}(P_j)}$. The proposition now follows from the fact that 
$\Delta_{\widetilde{\KK}_{\textrm{\textrm{dual}}}}$ and $\Delta_{q, \mathrm{up}}^{\KK, \LL}$ have the same minimal non-zero eigenvalue; see \cref{lemma:same_eig_val}.
\end{proof}

\begin{remark}
If $\LL$ is $(q+1)$-dimensional, then the condition that $B_{q+1}^\LL$ has full column rank (i.e., linearly independent columns) is equivalent to $H_{q+1}(\LL) = 0$. In particular, if $\LL$ is any $(q+1)$-dimensional simplicial complex embedded in $\mathbb{R}^{q+1}$, then $\LL$ is both $q$-non-branching and satisfies $H_{q+1}(\LL) = 0$. Hence, \cref{prop:realcheeger2} applies. 
\end{remark}

Combining \cref{prop:realcheeger1} and \cref{prop:realcheeger2} above together,
we form a Cheeger-type inequality with respect to the up persistent Laplacian $\Delta_{q, \textrm{up}}^{\KK, \LL}$. 
We summarize the result below:

\begin{theorem}\label{prop:realcheeger}
Let $\KK\hookrightarrow \LL$ where $\LL$ is unweighted and $q$-non-branching. 
Suppose the boundary matrix $B_{q+1}^\LL$ has full column rank and $B_{q+1}^{\LL, \KK}$ is not $0$.
Then,
\[
\min\limits_{j\leq h+r}\frac{\widehat{\mathbf{A}}(P_j)}{\mathbf{V}(P_j)} \leq \lambda_{\mathrm{min}}^{\KK, \LL}\leq \min\limits_{j\leq h+r} \frac{\mathbf{A}(P_j)}{\mathbf{V}(P_j)}.
\]
\end{theorem}
From this theorem, we obtain the following corollary.
\begin{corollary}
Let $\KK\hookrightarrow \LL$ where $\LL$ is unweighted and $q$-non-branching. 
Suppose the boundary matrix $B_{q+1}^\LL$ has full column rank and $B_{q+1}^{\LL, \KK}$ is not $0$. Then,
\[
\lambda_{\mathrm{min}}^{\KK, \LL}\leq q+2.
\]
\end{corollary}
\begin{proof}
    A $(q+1)$-simplex has precisely $q+2$ faces. In particular, if $f_{q}$ denotes the number of $q$-simplices in any simplicial complex, then $f_{q} \leq (q+2)f_{q+1}$ with equality precisely when you have a disjoint union of $(q+1)$-simplices. The bound follows readily. 
\end{proof}
To see that this bound is tight, let $\KK=\LL$ be equal to the standard $(q+1)$-simplex. Then, the matrix form of $\partial_{q+1}^{\KK, \LL}$ equals 
\[ \begin{bmatrix} 1 & -1 & 1 & \cdots & (-1)^{q+1} \end{bmatrix}^T.\]
Since $\mathbf{\widehat{A}}(P_1) = \mathbf{A}(P_1)$ in this case, $\lambda_{\mathrm{min}}^{\KK, \LL} = \mathbf{A}(P_1)/1 = q+2$ by \cref{prop:realcheeger}.

\begin{remark}
    For cubical complexes, one has $\lambda_{\mathrm{min}}^{\KK, \LL}\leq 2(q+1).$
\end{remark}

\subsection{An example}
\label{ex_1}
Consider the simplicial complexes $\KK\hookrightarrow \LL$ in \cref{fig_simplicial_cmplexes}. 
\begin{figure}[H]
  \centering
  \begin{subfigure}[b]{0.2\textwidth}
\begin{tikzpicture}
    \def\side{1.3}
    \filldraw[fill=yellow!40] (90: \side) -- (210: \side) -- (330: \side) -- cycle;
    \draw (0,0) -- (90: \side);
    \draw (0,0) -- (210: \side);
    \draw (0,0) -- (330: \side);

    \node at (90: \side + 0.3) {1};
    \node at (210: \side + 0.3) {2};
    \node at (330: \side + 0.3) {3};
    \node at (-0.2,0) [below right] {4}; 

    \fill (90: \side) circle (2pt);  
    \fill (210: \side) circle (2pt); 
    \fill (330: \side) circle (2pt); 
    \fill (0,0) circle (2pt);        
\end{tikzpicture}
\caption{$\LL$}%
\end{subfigure}%
\quad
\begin{subfigure}[b]{0.2\textwidth}
\begin{tikzpicture}
    \def\side{1.3}

    
    \draw (210: \side) -- (330: \side); 
    \draw (330: \side) -- (90: \side);  
    \draw (90: \side) -- (210: \side);

    \draw (0,0) -- (210: \side);
    \draw (0,0) -- (330: \side);

    \fill (90: \side) circle (2pt);  
    \fill (210: \side) circle (2pt); 
    \fill (330: \side) circle (2pt); 
    \fill (0,0) circle (2pt);        

    \node at (90: \side + 0.3) {1};
    \node at (210: \side + 0.3) {2};
    \node at (330: \side + 0.3) {3};
    \node at (-0.2,0) [below right] {4};
\end{tikzpicture}
\caption{$\KK$.}
\end{subfigure}
\caption{The simplicial complexes $\KK\hookrightarrow \LL$ associated with Example \ref{ex_1}.}%
\label{fig_simplicial_cmplexes}
\end{figure}
\paragraph{The Cheeger inequality}
We first assume that both $\KK$ and $\LL$ are unweighted.
The boundary operator $\partial_2^{\LL}: C_2^{\LL}\to C_1^{\LL}$ can be expressed as:
\begin{equation*}
B_2^{\LL}=
\begin{bNiceMatrix}[
  first-row,code-for-first-row=\scriptstyle,
  first-col,code-for-first-col=\scriptstyle,
]
& [124] & [143] & [234] \\
[12] & 1 & 0 & 0  \\
[24] & 1 & 0 & -1  \\
[14] & -1 & 1 & 0  \\
[13] & 0 & -1 & 0 \\
[34] & 0 & -1 & 1 \\
[23] & 0 & 0 & 1  \\
\end{bNiceMatrix},
\end{equation*}
and therefore
$
D_2^{\LL}
=
\begin{bNiceMatrix}[
  first-row,code-for-first-row=\scriptstyle,
  first-col,code-for-first-col=\scriptstyle,
]
& [124] & [143] & [234] \\
[14] & -1 & 1 & 0  \\
\end{bNiceMatrix}.
$
Hence, from \cref{prop:pers_lap}, 
we obtain that
\[
B_2^{\LL, \KK}
=
\begin{bNiceMatrix}[
  first-row,code-for-first-row=\scriptstyle,
  first-col,code-for-first-col=\scriptstyle,
]
 &  &  \\
[12]   & 1 & 0 \\
[24]   & 1 & -1 \\
[13]   & -1 & 0 \\
[34]   & -1 & 1 \\
[23]   & 0 & 1 \\
\end{bNiceMatrix}.
\]


Following the notation from \cref{prop:cheeger}, we get
$
S
=
\begin{bmatrix}
    2 & 0\\
    0 & 1
\end{bmatrix},$ and from \cref{thm:fastcomputation},
\[
\Delta_{1, \mathrm{up}}^{\KK, \LL}
=
B_2^{\LL, \KK} S^{-1} (B_2^{\LL, \KK})^T
=
\begin{bmatrix}
 1 & 0 \\
 1 & -1 \\
 -1 & 0 \\
 -1 & 1 \\
 0 & 1 \\
\end{bmatrix}
\begin{bmatrix}
    1/2 & 0\\
    0 & 1
\end{bmatrix}
\begin{bmatrix}
    1 & 1 & -1 & -1 & 0\\
    0 & -1 & 0 & 1 & 1
\end{bmatrix}.
\]
Moreover,
the up persistent Laplacian $\Delta_{1, \mathrm{up}}^{\KK, \LL}$ is equal to the Laplacian over the oriented hypergraph $\widetilde{\KK}$ as shown in \cref{fig_hg_ex_1}.
\begin{figure}[H]
    \centering
\begin{tikzpicture}
    \def\side{1.1}
    
    \begin{scope}[fill opacity=0.7]
    \filldraw[fill=yellow!40] (0,0) ellipse [x radius=2*\side, y radius=1.5*\side];

    \filldraw[fill=red!40] (0,-1.5*\side) ellipse [x radius=2*\side, y radius=1.5*\side];
    \end{scope}


    \fill (-0.5*\side, -0.8*\side) circle (2pt) node [above left] {$[24]+$}; 

    \fill (0.5*\side, -0.8*\side) circle (2pt) node [above right] {$[34]-$}; 

    \fill (-0.5*\side, -0.8*\side) circle (2pt) node [below left] {$[24]-$}; 

    \fill (0.5*\side, -0.8*\side) circle (2pt) node [below right] {$[34]+$}; 

    \fill (0.5*\side, 0.8*\side) circle (2pt) node [below right] {$[13]-$};

    \fill (-0.5*\side, 0.8*\side) circle (2pt) node [below left] {$[12]+$};

    \fill (0, -1.8*\side) circle (2pt) node [below] {$[23]+$};
\end{tikzpicture}
\caption{The weighted oriented hypergraph $\widetilde{\KK}$,
each vertex has weight $1$, 
the yellow hyperedge has weight $1/2$, 
and the red hyperedge has weight $1$.
The Laplacian $\Delta_0^{\widetilde{\KK}}$ is equal to the up persistent Laplacian $\Delta_{1, \mathrm{up}}^{\KK, \LL}$ with respect to $\KK\hookrightarrow \LL$ associated with Figure \ref{fig_simplicial_cmplexes}.}
\label{fig_hg_ex_1}
\end{figure}
Furthermore, the Laplacian of the dual hypergraph is 
\[
\Delta_{\widetilde{\KK}_{\textrm{\textrm{dual}}}}
=
S^{-1} (B_2^{\LL, \KK})^T B_2^{\LL, \KK} 
=
\begin{bmatrix}
    4/2 & -2/2\\
    -2 & 3
\end{bmatrix}.
\]
Notice that the diagonal entries of $\Delta_{\mathrm{dual}}$ are $4/2$ and $3$, measuring the ratio of the ``perimeter'' (number of edges) to the ``area'' (number of 2-simplices) of the polyhedra $[1243]$ and $[243]$,
\[
\frac{\mathbf{A}([1243])}{\mathbf{V}([1243])} = 4/2, \qquad \frac{\mathbf{A}([243])}{\mathbf{V}([243])}=\frac{3}{1}.
\]
Moreover,
\[
\frac{\widehat{\mathbf{A}}([1243])}{\mathbf{V}([1243])}=\frac{2}{2},
\quad
\frac{\widehat{\mathbf{A}}([243])}{\mathbf{V}([243])}=\frac{1}{1},
\]
hence the Cheeger-type inequality gives rise to
\[
1\leq \lambda_{\min}^{\KK, \LL}\leq 2.
\]
In fact,
by a direct computation we can obtain that $\lambda_{\min}^{\KK, \LL}=1$.

\section{Computing eigenvalues}
\label{sec:eigenvalues}
As mentioned in \cref{subsec:kron_red}, the eigenvalues of graph and Hodge Laplacians carry important topological and geometrical information. 
However, 
there are two drawbacks of computing the eigenvalues of the up persistent Laplacian $\Delta_{q, \text{up}}^{\KK, \LL}$ directly.
First,
it is computationally intensive. 
The construction of this matrix requires a quadratic number of operations in the number of simplices (\cref{prop:alg_lap}), and computing its full eigenvalue spectrum generally requires $O\left((n_q^{\KK})^3\right)$ operations. Furthermore, for moderately sized complexes, $\Delta_{q, \text{up}}^{\KK, \LL}$ can become unmanageably large and dense, rendering a spectral theory approach impractical for many applications. 
Second,
the up persistent Laplacian $\Delta_{q, \text{up}}^{\KK, \LL}$ relies on the weight matrices $W_{q+1}^{\LL, \KK}$ and $W_q^\KK$,
in practical application,
when the entries in $W_{q+1}^{\LL, \KK}$ and $(W_{q}^\KK)^{-1}$ are close to the machine epsilon
\footnote{The machine epsilon $\epsmach$ is identified to be the smallest positive number such that $1.0 + \epsmach\neq 1.0$ for a specific computer being used.}
for the device being used,
it may lead to erroneous results to compute eigenvalues directly from the up persistent Laplacian $\Delta_{q, \text{up}}^{\KK, \LL}$.
The readers can consult \cref{ex:numeric} for a concrete example.

To overcome these computational and memory barriers, a more efficient method is available that leverages the singular values of a related sparse matrix.
The method relies on the fact that the eigenvalues of
\[
    \Delta_{q, \text{up}}^{\KK, \LL} = B_{q+1}^{\LL, \KK} W_{q+1}^{\LL, \KK}\left(B_{q+1}^{\LL, \KK}\right)^{T}(W_{q}^\KK)^{-1}
\]
are identical to the eigenvalues of a symmetrized matrix $\tilde{\Delta}$. This is achieved via a similarity transformation:
$$ 
    \tilde{\Delta} = (W_q^{\KK})^{-1/2}\Delta_{q, \mathrm{up}}^{\KK, \LL}(W_q^{\KK})^{1/2} = (W_q^{\KK})^{-1/2} \left(B_{q+1}^{\LL, \KK} W_{q+1}^{\LL, \KK}\left(B_{q+1}^{\LL, \KK}\right)^{T}\right) (W_q^{\KK})^{-1/2}.
$$
This matrix can be factored as $\tilde{\Delta} = M M^T$, where the matrix $M$ is given by:
$$ 
    M = (W_q^{\KK})^{-1/2}B_{q+1}^{\LL, \KK} (W_{q+1}^{\LL, \KK})^{1/2}.
$$
Here the non-zero eigenvalues of $M M^T$ are the squares of the non-zero singular values of $M$.  This connection is computationally advantageous because the diagonal weight matrices do not alter the sparsity pattern of the boundary matrix, meaning $M$ is sparse.

The largest singular values of a sparse matrix $M$ can be computed efficiently using iterative methods such as the Golub-Kahan-Lanczos (GKL) algorithm; a generalization of the \emph{power method} to find the largest eigenvalue. After $k$ iterations, this procedure generates a $k \times k$ bidiagonal matrix whose singular values approximate $k$ singular values of $M$. Which $k$ singular values the algorithm approximates depends on the choice of starting vector and the distribution of singular values, but generically the largest singular value will be well-approximated. Importantly, the convergence rate of the GKL algorithm is not uniform across the spectrum; convergence to the largest singular values is much quicker, and they can thus be approximated with high accuracy using relatively few iterations. For a detailed treatment of the GKL algorithm, we refer the reader to~\cite{MR3024913, MR1792141} and \cref{appdx_GKL}. We summarize the relevant results below.

\begin{proposition}\label{prop:gkl}
Let $\KK\hookrightarrow \LL$ be a pair of simplicial complexes as described in \cref{thm:fastcomputation}. 
By iterating $k$ steps,
the GKL algorithm computes approximations to $k$ eigenvalues of $\Delta_{q, \text{up}}^{\KK, \LL}$ in 
\[
O\left((n_q^\LL-n_q^\KK)\times \alpha(n_{q+1}^\LL)+k\times (n_{q+1}^\LL+n_q^\KK)+k^2\right).
\]
\end{proposition}
\begin{proof}
The total complexity arises from three main steps. First, the cost of constructing the necessary components of the matrix $M$ is given by $(n_q^\LL-n_q^\KK)\times \alpha(n_{q+1}^\LL)$, as established in \cref{prop:alg_lap}. Second, the cost of computing a $k\times k$ bidiagonal matrix using GKL algorithm is bounded by $O(k\times (n_{q+1}^\LL+n_q^\KK))$ (see \cref{prop:GKL_speed}).
Finally,
according to \cite[Section 6.2]{MR1792141},
it costs $O(k^2)$ to compute the eigenvalues of a $k\times k$ bidiagonal matrix.
Hence the result follows.
\end{proof}

From \cref{cor:filt}, this result directly carries over to a simplex-wise filtration.

\begin{corollary}\label{cor:filt_complexity} 
Let $\KK_0\hookrightarrow\KK_1 \hookrightarrow\cdots\hookrightarrow \KK_m \hookrightarrow \LL$ be a $q$-non-branching simplicial filtration where each $\KK_i$ contains exactly one more $q$-simplex than $\KK_{i-1}$ for each $1\leq i\leq m$. 
The GKL algorithm computes approximations to $k$ eigenvalues of the matrices in $\{\Delta_{q, \text{up}}^{\KK_i, \LL}\}_{i=0}^{m}$ in
\[
O\Big(n_q^\LL\times \left((k+1)\times (n_{q+1}^\LL+n_q^\LL)+k^2\right)\Big)
\]
arithmetic operations.
In particular, the corresponding computation for all the matrices $\{\Delta_{q, \text{up}}^{\KK_i, \KK_j}\}_{i\leq j}$ can be done in
\[
O\Big((n_q^\LL)^2\times \left((k+1)\times (n_{q+1}^\LL+n_q^\LL)+k^2\right)\Big)
\]
arithmetic operations.
\end{corollary}
\begin{proof}
The result is obtained by computing all the matrices following  \cref{prop:gkl} and then applying \cref{cor:filt} to each matrix individually using the worst case where $n_q^\LL-n_q^{\KK_m}+m = n_q^\LL$ in the first expression and $(m+1)(n_{q+1}^\LL + n_q^{\KK_m}) = n_{q}^\LL(n_{q+1}^\LL + n_q^\LL)$ in the second expression, and $m+1=n_{q}^\LL$. This thus gives a computational time of
\[
O\Big(n_q^\LL\times \left(\alpha(n_{q+1}^\LL) + n_{q+1}^\LL + n_q^\LL\right)\Big).
\]
There are $O(n_q^\LL)$ matrices, and so the cost of computing the $k$ eigenvalues from \cref{prop:gkl} becomes 
\[
O\Big(n_q^\LL\times\left( \alpha(n_{q+1}^\LL)+n_{q+1}^\LL+n_q^\LL+k\times (n_q^\LL+n_{q+1}^\LL)+k^2\right)\Big).
\]
Adding these worst-case complexities together gives the desired result. The final bound is obtained by multiplication with the length of the filtration, which is at most $n_q^\LL$.
\end{proof}

It is well known~\cite[Section 2]{horak2013spectra} that the non-zero eigenvalues of the persistent Laplacian \[\Delta_{q}^{\mathcal{K}, \mathcal{L}} = \Delta_{q, \mathrm{up}}^{\mathcal{K}, \mathcal{L}} + \Delta_{q, \mathrm{down}}^{\mathcal{K}}\]
are given by the union (counted with multiplicity) of the eigenvalues of the up Laplacian and the down Laplacians. This is immediate from 
\[\Delta_{q, \mathrm{up}}^{\mathcal{K}, \mathcal{L}}\circ \Delta_{q, \mathrm{down}}^{\mathcal{K}} = \Delta_{q, \mathrm{down}}^{\mathcal{K}}\circ \Delta_{q, \mathrm{up}}^{\mathcal{K}, \mathcal{L}}=0.\]
Since the multiplicity of the $0$ eigenvalue in $\Delta_{q}^{\mathcal{K}, \mathcal{L}}$ equals the $q$-th persistent Betti number of the inclusion $\KK\hookrightarrow \LL$, information already captured by persistent homology, the non-zero spectrum is likely of more interest to the practitioner.

The matrix representation of the down Laplacian can be written as $(W_q^\KK)^{-1}\left(B^{\KK}_{q}\right)^T W_{q-1}^{\KK} B^{\KK}_{q}$. Following the same symmetrization procedure as above, its eigenvalues can be computed from the singular values of a  matrix derived from $B^{ \KK}_{q}$.

The boundary matrix $B^{\KK}_{q}$ is sparse, with exactly $(q+1)$ non-zero entries per column corresponding to the faces of each $q$-simplex.
Using the GKL algorithm we can approximate $k$ eigenvalues of the down Laplacian in 
\[
O\left(k\times (n_q^\KK\times (q+1)+n_{q-1}^\KK)+k^2\right)
\]
arithmetic operations for a fixed complex $\mathcal{K}$. The total time for all simplicial complexes in a full filtration from $\emptyset$ to $\mathcal{L}$ is bounded by
\[
O\Big(n_q^\LL\times\left(k\times (n_q^\KK\times (q+1)+n_{q-1}^\KK)+k^2\right)\Big).
\]
This can potentially offer a significant improvement over methods requiring explicit matrix formation and dense eigensolvers, making the computation of a few eigenvalues a practical addition to persistent homology. We have included some computations in \cref{fig:image_eig_sing} for the $k$ smallest eigenvalues and for the $k$ largest eigenvalues. It is clear that this SVD-based approach is substantially more efficient than methods that first form the dense Laplacian matrix product.

\begin{remark}
\label{rem:hodge-efficient}
An insightful method for visualizing topological evolution throughout a filtration involves plotting the barcode alongside a curve that tracks the smallest non-zero eigenvalue of the standard (non-persistent) Hodge Laplacian at each stage (see, e.g., \cite[Figure 5]{wei2023}). Since the Hodge Laplacian is the sum of up and down Laplacians---whose top eigenvalues can be efficiently computed using the sparse matrix techniques as done for the complex $\KK$ above---this visualization approach becomes practically generalizable to arbitrary (not necessarily non-branching) filtrations for large eigenvalues. This extension is a promising direction for future work.
\end{remark}

\section{Cubical complexes and X-Ray images}\label{subsec:cubical}

Cubical complexes provide an efficient framework for representing image data in TDA. 
The persistent Laplacian with respect to cubical complexes has been successfully applied in image analysis \cite{davies23c}, 
as pixels (2D images) and voxels (3D images) are naturally represented using $2$-cubes (squares) and $3$-cubes (cubes), respectively. 
Cubical complexes are automatically non-branching and have a natural oriented matrix representation of the boundary operator, 
making it unnecessary to check matrix orientability.
In this section, we first present a simple example to demonstrate the computation of the persistent Laplacian in the context of cubical complexes (\cref{ex_cub_1}),
then we experiment with this framework on an X-Ray image.
We refer to \cite{MR3025945} and \cref{appdx_cub_cplx} for more details on cubical complexes.
\subsection{A worked-out example}

\begin{figure}[H]
  \centering
  \begin{subfigure}[b]{0.3\textwidth}
\begin{tikzpicture}
    \fill[red!50] (0,1) rectangle (1,2);
        \fill[red!50] (1,1) rectangle (2,2);
        \fill[red!50] (0,0) rectangle (1,1);
        \fill[red!50] (1,0) rectangle (2,1);
    \draw[step=1cm, gray, very thin] (0,0) grid (2,2);
    
    \foreach \x in {0,1,2} {
        \foreach \y in {0,1,2} {
            \node[fill=black, circle, inner sep=1pt] at (\x,\y) {};
        }
    }
    \node[anchor=south east] at (0,0) {7};
    \node[anchor=south west] at (1.65,0) {9};
    \node[anchor=north east] at (0.05,2.5) {1};
    \node[anchor=north west] at (1.65,2.5) {3};
    \node[anchor=north] at (0.85,2.5) {2};
    \node[anchor=south] at (0.85,0) {8};
    \node[anchor=west] at (1.65,1.3) {6};
    \node[anchor=east] at (0.05,1.3) {4};
    \node[anchor=center] at (0.85,1.3) {5};
\end{tikzpicture}
\caption{$\LL$.}%
\end{subfigure}%
   \quad
   \begin{subfigure}[b]{0.3\textwidth}
   \begin{tikzpicture}

        \draw[gray, very thin] (0,0) -- (0,2);
        \draw[gray, very thin] (1,0) -- (1,1);
        
        \draw[gray, very thin] (2,1) -- (2,2);
        
        \draw[gray, very thin] (0,0) -- (1,0);
        \draw[gray, very thin] (1,1) -- (2,1);
        \draw[gray, very thin] (0,1) -- (1,1);
        \draw[gray, very thin] (1,2) -- (2,2);

        \draw[gray, very thin] (1,0) -- (2,0);
        \draw[gray, very thin] (2,0) -- (2,1);
        
        \draw[gray, very thin] (0,2) -- (1,2);
        
        \foreach \x in {0,1,2} {
            \foreach \y in {0,1,2} {
                \ifnum\x=2\relax
                    \ifnum\y=0\relax
                    \else
                        \node[fill=black, circle, inner sep=1pt] at (\x,\y) {};
                    \fi
                \else
                    \node[fill=black, circle, inner sep=1pt] at (\x,\y) {};
                \fi
            }
        }
        
        \node[anchor=south east] at (0,0) {7};
    \node[anchor=north east] at (0.05,2.5) {1};
    \node[anchor=north west] at (1.65,2.5) {3};
    \node[anchor=north] at (0.85,2.5) {2};
    \node[anchor=south] at (0.85,0) {8};
    \node[anchor=west] at (1.65,1.3) {6};
    \node[anchor=east] at (0.05,1.3) {4};
    \node[anchor=center] at (0.85,1.3) {5};
    \node[anchor=south west] at (1.65,0) {9};
    \node[fill=black, circle, inner sep=1pt] at (2,0) {};
    \end{tikzpicture}
    \caption{$\KK$.}
\end{subfigure}%

\caption{The cubical complexes $\KK \hookrightarrow \LL$ associated with Example \ref{ex_cub_1}.}
\label{fig_cubical_cmplexes}
\end{figure}

\begin{example}\label{ex_cub_1}
Let $\KK\hookrightarrow\LL$ be the pair of cubical complexes in \cref{fig_cubical_cmplexes}. The boundary operator $\partial_2^{\LL}$ and the submatrix $D_2^{\LL}$ can be expressed as
\begin{equation*}
B_2^{\LL}=
\begin{bNiceMatrix}[
  first-row,code-for-first-row=\scriptstyle,
  first-col,code-for-first-col=\scriptstyle,
]
& [4512] & [5623] & [7845] & [8956] \\
[12] & -1 & 0 & 0 & 0 \\
[23] & 0 & -1 & 0 & 0 \\
[41] & -1 & 0 & 0 & 0 \\
[52] & 1 & -1 & 0 & 0 \\
[63] & 0 & 1 & 0 & 0 \\
[45] & 1 & 0 & -1 & 0 \\
[56] & 0 & 1 & 0 & -1 \\
[74] & 0 & 0 & -1 & 0 \\
[85] & 0 & 0 & 1 & -1 \\
[96] & 0 & 0 & 0 & 1 \\
[78] & 0 & 0 & 1 & 0 \\
[89] & 0 & 0 & 0 & 1 \\
\end{bNiceMatrix},
\quad \quad
D_2^{\LL}
=
\begin{bNiceMatrix}[
  first-row,code-for-first-row=\scriptstyle,
  first-col,code-for-first-col=\scriptstyle,
]
& [4512] & [5623] & [7845] & [8956] \\
[52] & 1 & -1 & 0 & 0 \\
\end{bNiceMatrix}
\end{equation*}
respectively.
Thus according to \cref{prop:pers_lap},
\[
B_2^{\LL, \KK}
=
\begin{bNiceMatrix}[
  first-row,code-for-first-row=\scriptstyle,
  first-col,code-for-first-col=\scriptstyle,
]
 &  & &  \\
[12]  & -1 & 0 & 0 \\
[23]  & -1 & 0 & 0 \\
[41]  & -1 & 0 & 0 \\
[63]  & 1 & 0 & 0 \\
[45]  & 1 & -1 & 0 \\
[56]  & 1 & 0 & -1 \\
[74]  & 0 & -1 & 0 \\
[85]  & 0 & 1 & -1 \\
[96]  & 0 & 0 & 1 \\
[78]  & 0 & 1 & 0 \\
[89]  & 0 & 0 & 1 \\
\end{bNiceMatrix}.
\]
Since the cardinality of $C_1, C_2, C_3$ are $2$, $1$ and $1$ respectively,
we get from \cref{thm:fastcomputation} that the diagonal matrix $S$ (as in \cref{prop:cheeger}) is
\[
S
=
\begin{bmatrix}
2 & 0 & 0\\
0 & 1 & 0\\
0 & 0 & 1\\
\end{bmatrix},
\]
and the up persistent Laplacian equals
\[
\Delta_{1, \mathrm{up}}^{\KK, \LL}
=
B_2^{\LL, \KK} S^{-1} \left(B_2^{\LL, \KK}\right)^T.
\]
\end{example}

\begin{figure}[H]
\centering
\begin{tikzpicture}
    \node (v12) at (1.5,3.5) {};
    \node (v23) at (3.5, 3.5) {};
    \node (v41) at (0, 3) {};
    \node (v63) at (5,3) {};
    \node (v45) at (1.5,2.5) {};
    \node (v56) at (3.5,2.5) {};
    \node (v74) at (0,1.5) {};
    \node (v85) at (2.5,1.5) {};
    \node (v96) at (5,1.5) {};
    \node (v78) at (1.5,0.5) {};
    \node (v89) at (3.5,0.5) {};

    \begin{scope}[fill opacity=0.7]
     \filldraw[fill=yellow!40] (2.5, 3.1) ellipse [x radius=3.4, y radius=1.1];
    \filldraw[fill=red!40] (1.3, 1.5) ellipse [rotate=0, x radius=1.5, y radius=1.4];
     \filldraw[fill=blue!40] (3.7, 1.5) ellipse [rotate=0, x radius=1.5, y radius=1.4];
    \end{scope}
    
    \fill (v12) circle (0.1) node [above] {$[12]-$};
    \fill (v23) circle (0.1) node [above] {$[23]-$};
    \fill (v41) circle (0.1) node [above] {$[41]-$};
    \fill (v63) circle (0.1) node [above] {$[63]+$};
    \fill (v45) circle (0.1) node [above=0.2cm] {$[45]+$};
    \fill (v45) circle (0.1) node [below] {$[45]-$};
    \fill (v56) circle (0.1) node [above=0.2cm] {$[56]+$}; 
    \fill (v56) circle (0.1) node [below] {$[56]-$};
    \fill (v74) circle (0.1) node [below right] {$[74]-$};
    \fill (v85) circle (0.1) node [left=0.1cm] {$[85]+$};
    \fill (v85) circle (0.1) node [right=0.1cm] {$[85]-$};
    \fill (v96) circle (0.1) node [left] {$[96]+$};
    \fill (v78) circle (0.1) node [above] {$[78]+$};
    \fill (v89) circle (0.1) node [above] {$[89]+$};
\end{tikzpicture}
\caption{\footnotesize{The oriented hypergraph $\widetilde{\KK}$,
which contains three hyperedges $e_1, e_2, e_3$,
with $e_1=\{[12]-, [23]-, [41]-, [63]+, [45]+, [56]+\}$,
$e_2=\{[45]-, [74]-, [85]+, [78]+\}$,
and
$e_3=\{[56]-, [85]-, [96]+, [89]+\}$.
}}
\label{fig_hyp_grph_1}
\end{figure}
We also obtain that $\Delta_{1, \mathrm{up}}^{\KK, \LL}$ is equivalent to the Laplacian $\Delta_{0}^{\widetilde{\KK}}$ of the weighted hypergraph $\widetilde{\KK}_3$ given in \cref{fig_hyp_grph_1}. The weight of the yellow hyperedge on top is $1/2$, and the weights of the red hyperedge (bottom left) and the blue hyperedge (bottom right) are both equal to $1$.

\subsection{Experiments on a chest X-Ray image}
We conduct experiments on a $224\times 224$ chest X-ray image randomly selected from the ChestMNIST dataset \cite{medmnistv1, medmnistv2}. To analyze the performance of our algorithm across different scales, we generate a multi-resolution set of images by iteratively applying a $2\times 2$ max-pooling operation to the original image (see \cref{fig:image_x_ray}, top row).

From each of these four images, we derive two cubical complexes by thresholding pixel intensities. First, we construct the complexes $\KK^1, \KK^2, \KK^3, \KK^4$ from pixels with values less than $50$ (\cref{fig:image_x_ray}, middle row). Second, we construct a set of larger complexes, $\LL^1, \LL^2, \LL^3, \LL^4$, using a higher threshold of $150$ (\cref{fig:image_x_ray}, bottom row). This process yields four pairs of nested complexes $\KK^j\hookrightarrow \LL^j$ for $j=1,2,3,4$. For each pair, we define a cube-wise filtration $\KK^j=\KK_0^j\hookrightarrow \KK_1^j\hookrightarrow \cdots \hookrightarrow \KK^j_m=\LL^j$, where each step $\KK^j_{i} \hookrightarrow \KK^j_{i+1}$ corresponds to the addition of a single $1$-cube. The total number of added $1$-cubes, $m$, for the four filtrations are $47211$, $11843$, $2954$, and $702$, respectively.

\begin{figure}[ht]
    \centering
    \setlength{\tabcolsep}{5.5pt} 
    \begin{tabular}{cccc}
        \includegraphics[width=0.22\textwidth]{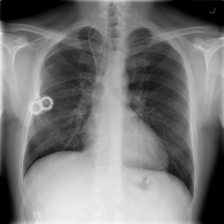} &
        \includegraphics[width=0.22\textwidth]{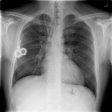} &
        \includegraphics[width=0.22\textwidth]{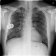} &
        \includegraphics[width=0.22\textwidth]{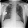} \\
        \multicolumn{4}{c}{\footnotesize{(a) Original image and three downscaled versions via max-pooling}.} \\[6pt]

        \includegraphics[width=0.22\textwidth]{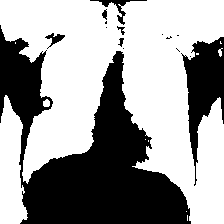} &
        \includegraphics[width=0.22\textwidth]{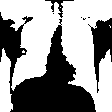} &
        \includegraphics[width=0.22\textwidth]{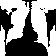} &
        \includegraphics[width=0.22\textwidth]{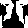} \\
        \multicolumn{4}{c}{\footnotesize{(b) Images from pixels with values $< 50$ (defining complexes $\KK^j$)}.} \\[6pt]
        
        \includegraphics[width=0.22\textwidth]{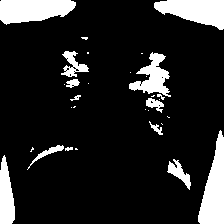} &
        \includegraphics[width=0.22\textwidth]{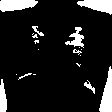} &
        \includegraphics[width=0.22\textwidth]{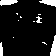} &
        \includegraphics[width=0.22\textwidth]{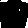} \\
        \multicolumn{4}{c}{\footnotesize{(c) Images from pixels with values $< 150$ (defining complexes $\LL^j$)}.} \\[6pt]
    \end{tabular}
    \caption{The image data processing pipeline. (a) The original $224 \times 224$ X-ray image and three downscaled versions. (b) Binary images used to define the initial cubical complexes $\KK^j$. (c) Binary images used to define the larger terminal complexes $\LL^j$.}
    \label{fig:image_x_ray}
\end{figure}

\subsubsection{The speed of computing $B_1^{\LL^j, \KK_n^j}$ and $\Delta_{1, \text{up}}^{\KK^j_n, \LL^j}$}\label{subsec:compute_speed}
We first evaluate\footnote{{\color{blue}The code is available on \hyperlink{https://github.com/ruidongsmile/UpPersLapCubicalComplex}{https://github.com/ruidongsmile/UpPersLapCubicalComplex}.}}
the performance of our DSU based method from \cref{sec:filtration} against the Schur complement approach of \cite[Section 5]{pers_lap} for computing the up persistent Laplacian, $\Delta_{1, \text{up}}^{\KK^j_n, \LL^j}$, along each filtration. We also time the computation of the constituent matrices $\left(B_1^{\LL^j, \KK_n^j}, W_1^{\LL^j, \KK^j_n}\right)$ and their product $B_1^{\LL^j, \KK_n^j}\big(W_1^{\LL^j, \KK^j_n}\big)^{1/2}$. The results are summarized in \cref{fig:image_plots}. For larger complexes (e.g., $\KK^1 \hookrightarrow \LL^1$), our DSU-based method is notably faster. Conversely, for smaller complexes, the Schur complement method exhibits better performance. In all cases, the direct computation of constituent matrices and their product is significantly faster than computing the full persistent Laplacian.

\begin{figure}[ht]
    \centering
    \setlength{\tabcolsep}{0.1pt}
    \begin{tabular}{cc}
         \subcaptionbox{The filtration $\KK^1_0\hookrightarrow \cdots \hookrightarrow\LL^1$.\label{fig:img1}}{\includegraphics[width=0.5\textwidth]{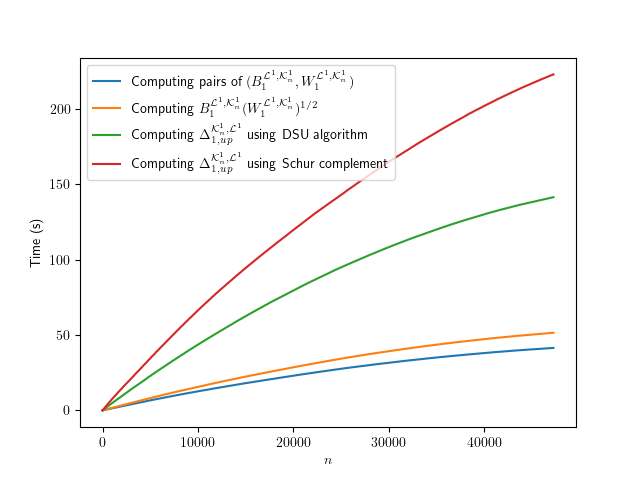}} &
         \subcaptionbox{The filtration $\KK^2_0\hookrightarrow \cdots \hookrightarrow\LL^2$.\label{fig:img2}}{\includegraphics[width=0.5\textwidth]{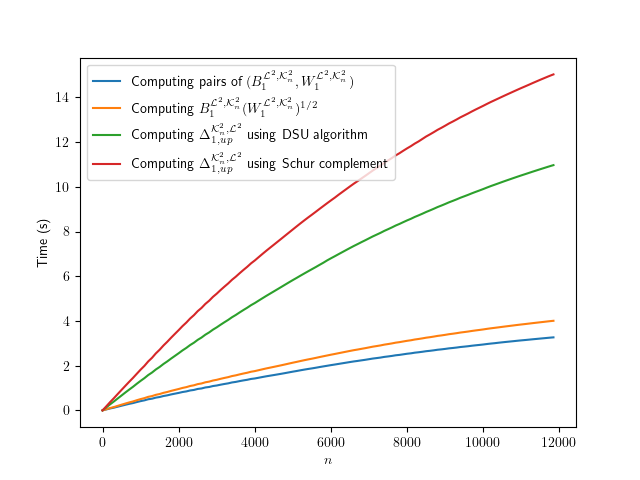}} \\
         \subcaptionbox{The filtration $\KK^3_0\hookrightarrow \cdots \hookrightarrow\LL^3$.\label{fig:img3}}{\includegraphics[width=0.5\textwidth]{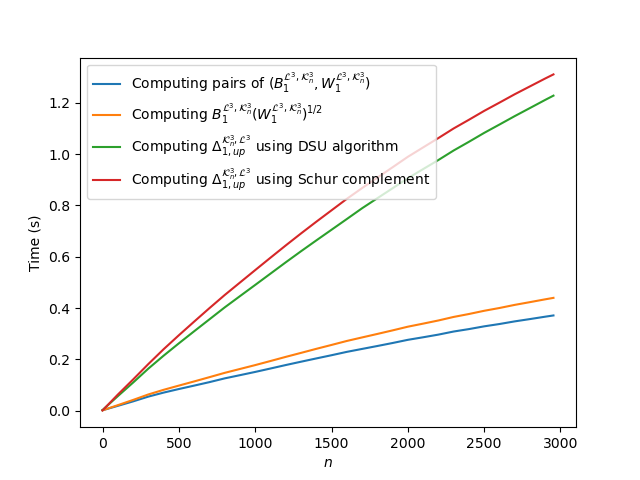}} &
         \subcaptionbox{The filtration $\KK^4_0\hookrightarrow \cdots \hookrightarrow\LL^4$.\label{fig:img4}}{\includegraphics[width=0.5\textwidth]{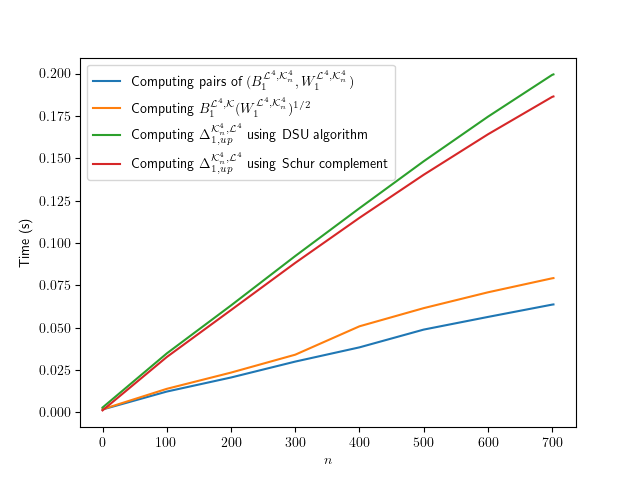}}
    \end{tabular}
    \caption{Computation time analysis for the four filtrations. Each plot compares the time taken by our DSU-based method, the Schur complement method to compute the up persistent Laplacians, and the computation of intermediate matrices. Time is plotted against the number of $1$-cubes added to the filtration. }
    \label{fig:image_plots}
\end{figure}

\begin{figure}[ht]
    \centering
    \setlength{\tabcolsep}{0.1pt}
    \begin{tabular}{cc}
         \subcaptionbox{The filtration $\KK^1_0\hookrightarrow \cdots \hookrightarrow\LL^1$.\label{fig:img5}}{\includegraphics[width=0.5\textwidth]{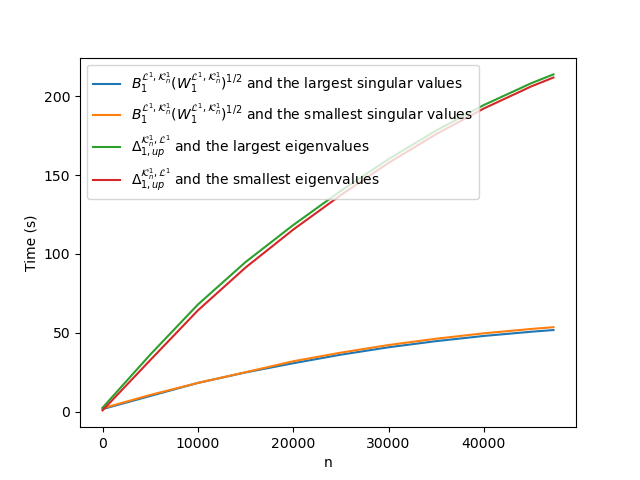}} &
         \subcaptionbox{The filtration $\KK^2_0\hookrightarrow \cdots \hookrightarrow\LL^2$.\label{fig:img6}}{\includegraphics[width=0.5\textwidth]{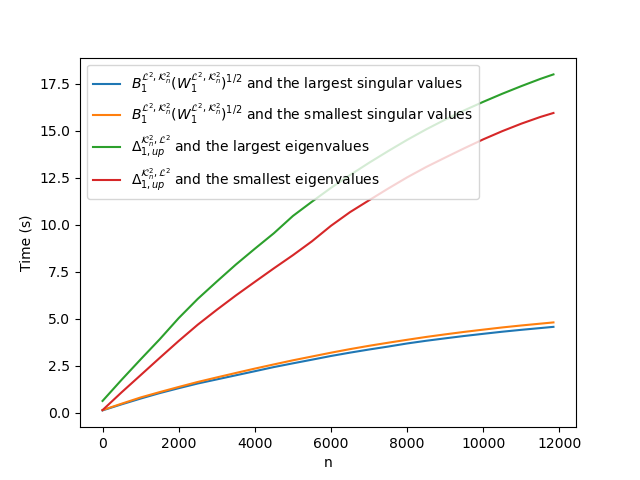}} \\
         \subcaptionbox{The filtration $\KK^3_0\hookrightarrow \cdots \hookrightarrow\LL^3$.\label{fig:img7}}{\includegraphics[width=0.5\textwidth]{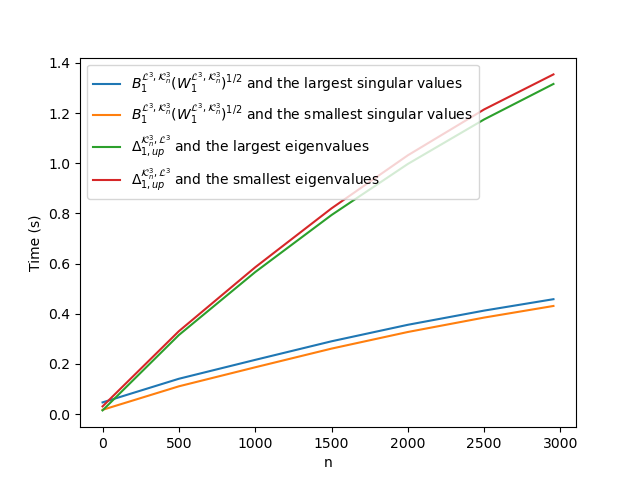}} &
         \subcaptionbox{The filtration $\KK^4_0\hookrightarrow \cdots \hookrightarrow\LL^4$.\label{fig:img8}}{\includegraphics[width=0.5\textwidth]{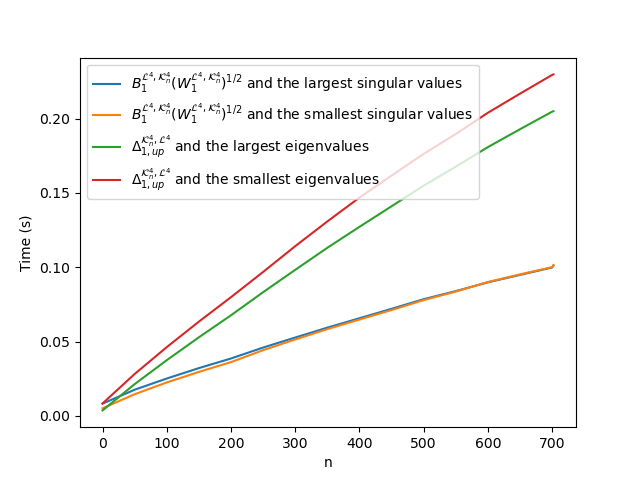}}
    \end{tabular}
    \caption{Comparison of computation times for largest/smallest eigenvalues and singular values as described in \cref{subsec:compute_speed}.}
    \label{fig:image_eig_sing}
    \end{figure}

Next, we compare the computational cost of spectral decomposition; see \cref{sec:eigenvalues}.

For each filtration 
\[\KK_0^j\hookrightarrow \cdots \hookrightarrow \LL^j, \qquad j=1, 2, 3, 4,\]
we first compute the up persistent Laplacian $\Delta_{1, \text{up}}^{\KK^j_n, \LL^j}$ using the Schur complement method, 
together with the largest and smallest non-zero eigenvalue. Then, we compute the matrix $B_1^{\LL^j, \KK_n^j}\big(W_1^{\LL^j, \KK^j_n}\big)^{1/2}$ and its largest and smallest singular value. We compute these values at every $5000$-th $(j=1)$, $500$-th $(j=2)$, $500$-th $(j=3)$ and $50$-th $(j=4)$ step of the filtration. As shown in \cref{fig:image_eig_sing},
calculating the singular values is substantially faster than calculating the eigenvalues of the corresponding Laplacian\footnote{Computations were performed using the \texttt{sparse.linalg.eigsh} and \texttt{sparse.linalg.svds} functions in SciPy.}.

\section{Discussion and further work}
\label{sec:discussion}
The persistent Laplacian has already demonstrated superior performance compared to persistent homology on image data \cite{davies23c}. We believe that the contributions of this paper makes the persistent Laplacian a practical tool for image data, e.g., by including the top $k$ eigenvalues.  Exploring the discriminatory power of such enhanced descriptors is a key direction for our future work (see also \cref{rem:hodge-efficient}). Another related and hitherto open problem is how to efficiently work with such a large number of eigenvalues in machine learning applications.

On the theoretical side, this paper introduces a Cheeger-type inequality for the persistent Laplacian. While the minimal non-zero eigenvalue, $\lambda_{\mathrm{min}}^{\KK, \LL}$, of the up persistent Laplacian is known to capture geometric information, its stability under small perturbations of the filtration remains an important open question that warrants further investigation.

\FloatBarrier
\bibliography{sample}

\appendix

\section{Graphs and incidence matrices}
In this section,
we describe the weighted graphs used in \cref{subsec:singular_mtx}.

\begin{definition}
A directed (resp. undirected) graph $G=(V, E)$ is a discrete structure where the vertex set $V$ is a finite set and the edge set $E$ is a set whose each element is an (resp. ordered) unordered pair of distinct vertices in V.
\end{definition}

\begin{definition}
We say $G=(V, E)$ is an undirected graph with loops if each element $e\in E$ is either an unordered pair of distinct vertices or a single vertex in $V$.
\end{definition}

\begin{definition}\label{def:incd_mtx_undir}
The incidence matrix $I(G)$ of an undirected graph (with loops) $G=(V, E)$ is a $\{0, 1\}$-matrix whose rows and columns are indexed by the vertices and edges of $G$ respectively,
such that the $(v, e)$-entry of $G$ is equal to $1$ if $v\in e$, 
otherwise it is equal to $0$.
\end{definition}

\begin{remark}
Usually the incidence matrix with respect to an undirected graph $G$ with loops is a $\{0, 1, 2\}$-matrix,
i.e., 
each loop is indexed by $2$ in $I(G)$,
this is different from \cref{def:incd_mtx_undir}.
\end{remark}

We will then give the formal definition of the incidence matrix with respect to a directed graph $G=(V, E)$.
We refer to an ordered pair of adjacent vertices as an \emph{arc} of $G$.
An orientation of $G$ is defined as a function $\sigma$ from the arcs of $G$ to $\{-1, 1\}$ such that if $(u, v)$ is an arc,
then $\sigma(u, v)=-\sigma(v, u)$.
We regard the edge $[uv]$ as directed from tail $u$ to head $v$ if $\sigma(u, v)=1$.
A directed graph $G$ is associated with a selected orientation $\sigma$.

\begin{definition}
The incidence matrix $I(G)$ of a directed graph $G=(V, E)$ is the $\{0, \pm 1\}$-matrix with rows and columns indexed by vertices and edges of $G$ respectively such that the $(u, e)$-entry is equal to $1$ if $u$ is the head of $e$,
 $-1$ if $u$ is the tail of $e$,
 and $0$ otherwise.
\end{definition}

\section{Pseudocodes associated with Union-Find (DSU) algorithm}\label{appdx_codes}

\begin{algorithm}[H]
\caption{\textsc{DSU}}
\label{alg:dsu}
\begin{algorithmic}[1]
\State \textbf{Input:} $|D|$ 
\Comment{A $k\times l$ non-branching matrix}
\State \textbf{Output:} parent

\State $\text{parent} \gets [0, 1, \cdots, l]$ 
\State $\text{rank} \gets [1, 1, \cdots, 1]$ \Comment{$l$ copies of $1$s.}
\State $\text{edge}\gets 1$
\While{$\text{edge} \leq k$}
\State $\text{nodes} \gets \text{column indices of nonzero entries in } D(\text{edge, :})$
\If{$\text{len(nodes)} = 2$}
\State $x, y \gets \text{nodes}$
\State $\textsc{Union}(x, y)$
\EndIf  
\State $\text{edge} \gets \text{edge}+1$
\EndWhile


\State $\text{node}\gets 1$
\While{$\text{node} \leq l$}
\State $\text{root}\gets \textsc{Find(node)}$
\If{$\text{parent}[\text{node}] \neq \text{root}$}
\State $\text{parent}[\text{node}]\gets \text{root}$
\EndIf
\State $\text{node}\gets \text{node} + 1$
\EndWhile

\end{algorithmic}
\end{algorithm}

\begin{algorithm}[H]
\caption{\textsc{Union}}
\label{alg:union}
\begin{algorithmic}[1]
\State \textbf{Input:} $x$, $y$
\Comment{$x$, $y$ are two column indices of $|D|$}

\State $\text{root}_x \gets \textsc{Find}(x)$
    \State $\text{root}_y \gets \textsc{Find}(y)$
    \If{$\text{root}_x \neq \text{root}_y$}
        \If{$\text{rank}[\text{root}_x] > \text{rank}[\text{root}_y]$}
            \State $\text{parent}[\text{root}_y] \gets \text{root}_x$
        \Else
            \State $\text{parent}[\text{root}_x] \gets \text{root}_y$
            \If{$\text{rank}[\text{root}_x] = \text{rank}[\text{root}_y]$}
                \State $\text{rank}[\text{root}_y] \gets \text{rank}[\text{root}_y] + 1$
            \EndIf
        \EndIf
    \EndIf

\end{algorithmic}
\end{algorithm}

\begin{algorithm}[H]
\caption{\textsc{Find}}
\label{alg:find}
\begin{algorithmic}[1]
\State \textbf{Input:} $x$ 
\State \textbf{Output:} root of $x$

    \While{$\text{parent}[x] \neq x$}
        \State $\text{parent}[x] \gets \text{parent}[\text{parent}[x]]$ \Comment{Path compression}
        \State $x \gets \text{parent}[x]$
    \EndWhile
    \State \Return $x$

\end{algorithmic}
\end{algorithm}

\section{Oriented hypergraphs}\label{appdx_hg}
In this section, we recall the basics of hypergraphs, mostly following the exposition in \cite{MR4236395}.

\begin{definition}[\cite{MR4236395}]
An \emph{oriented hypergraph} is a pair $H=(V, E)$ such that $V$ is a finite set of vertices and $E$ is a set such that each edge $e\in E$ is a pair of disjoint elements $(e_{\mathrm{in}}, e_{\mathrm{out}})$ (input and output) in the power set $\mathcal{P}(V)$.
\end{definition}

\begin{remark}
For an edge $e=(e_{\mathrm{in}}, e_{\mathrm{out}})$,
the set $e_{\mathrm{in}}$ or $e_{\mathrm{out}}$ can be empty.
\end{remark}

\begin{definition}[\cite{MR4236395}]
An oriented hypergraph $H$ is \emph{connected} if for each pair of vertices $u, v\in V$,
there exist $v_1, \cdots, v_k\in V$ and $e_1, \cdots e_{k-1}\in E$ such that $v_1=u, v_k=v$,
and $\{v_i, v_{i+1}\subset e_i\}$ for each $1\leq i\leq k-1$.
In other words,
there exists a path which connects $u$ and $v$.
\end{definition}

\begin{definition}[\cite{MR4236395}]
We say an oriented hypergraph $H=(V, E)$ has $r$ connected components if there exists $H_j=(V_j, E_j)$ for $1\leq j\leq r$ such that:
\begin{enumerate}
    \item $H_j$ is a connected hypergraph for each $1\leq j\leq r$.
    \item For $1\leq i, j\leq r$ if $i\neq j$, 
    then $V_i\cap V_j=\emptyset$ and $E_i\cap E_j=\emptyset$.
    \item $\bigcup V_j = V$,
    and $\bigcup E_j=E$.
\end{enumerate}
\end{definition}

\begin{definition}[\cite{MR4236395}]
 If the cardinality of $V$ is $n$,
 the cardinality of $E$ is $m$,
then the \emph{incidence matrix} of $H=(V, E)$ is $\mathcal{I}\coloneqq (\mathcal{I}_{ve})_{i\in V, e\in E}$ with 
 \[
 \mathcal{I}_{ve}
 \coloneqq
 \begin{cases}
1 & \textrm{if } v\in e_{\mathrm{in}},\\
-1 & \textrm{if } v\in e_{\mathrm{out}},\\
0 & \textrm{else}.
 \end{cases}
 \]
\end{definition}

The following definition is novel to this paper, and introduces the concept of an \emph{weighted} oriented hypergraph. 

\begin{definition}
Let $W_0^v$ be a positive weight over the vertex $v\in V$,
and $W_1^e$ be a positive weight over the hyperedge $e\in E$.
We define a weighted oriented hypergraph as a 
quadruple $(V, E, W_0, W_1)$,
where $W_0=\{W_0^v: v\in V\}$,
$W_1=\{W_1^e: e\in E\}$.
\end{definition}

Inspired by the Laplacian for oriented hypergraphs given in \cite{MR4236395}, we arrive at the following definition of a Laplacian over a weighted oriented hypergraph $(V, E, W_0, W_1)$.

\begin{definition}
\label{def:laphyper}
Let $(V, E, W_0, W_1)$ be a weighted oriented hypergraph.
The Laplacian $L_0$ over $V$ is 

\[
\begin{aligned}
L_0([v])
\coloneqq
&\sum_{e_{\textrm{in}:} v \textrm{ input}} W^e_1 \left(W_0^{v'} \sum_{v' \textrm{:input of} e_{\mathrm{in}}} [v']  
-
W_{0}^{w'}\sum_{w'\textrm{:output of } e_{\textrm{in}}}[w']
\right)\\
&-
\sum_{e_{\mathrm{out}:} v\, \mathrm{output}} W_{1}^e \left( W_{0}^{\hat{v}}\sum_{\hat{v}\mathrm{: input \, of \,} e_{\mathrm{out}}}[\hat{v}] 
-
W_0^{\hat{v}}\sum_{\hat{v}\mathrm{: output \, of \,} e_{\mathrm{out}}}[\hat{w}]
\right).
\end{aligned}
\]

Equivalently,
we can write the Laplacian $L_0$ in the matrix form as
\[
L_0
=
W_0\, \mathcal{I}\, W_1\, \mathcal{I}^T,
\]
where $\mathcal{I}$ is the incidence matrix of the underlying oriented hypergraph.
\end{definition}

\begin{remark}
\cref{def:laphyper} is a generalization of the Laplacian for directed graphs.
In fact,
for a directed graph $(V, E)$,  
if we assign 
$W_0=\id_n$ and $W_1=\id_m$,
then $L_0=\mathcal{I}\mathcal{I}^T$ gives the algebraic Laplacian used in spectral graph theory \cite{MR1421568}.
Furthermore,
if $(V, E)$ is an oriented hypergraph,
and
\[W_0=\diag\left(1/\deg v_1,\cdots, 1/\deg v_n \right)\] where $\deg v_i$ denotes the degree of vertex $v_i$ for $i\leq n$,
and assign $W_1=\id_{m}$,
then $L_0$ is equal to the normalized Laplacian over the oriented hypergraph $(V, E)$; we refer to {\cite[Section 1.4]{MR4433789}} for more details.
   
\end{remark}

\begin{definition}[\cite{MR4236395}]\label{def_dual_hg}
Let $H=(V, E)$ be an oriented hypergraph with $V=\{v_1,\cdots, v_n\}$ and $E=\{e_1,\cdots, e_m\}$.
We define the \emph{dual hypergraph} $H^T$ with respect to $H$ as $H^T=(V', H')$,
where $V'=\{v'_1, \cdots, v'_m\}$ and $E'=\{e'_1, \cdots, e'_n\}$,
\[
v'_j\in e'_i \textrm{ as input (resp. output)} 
\Longleftrightarrow
v_i\in e_j \textrm{ as input (resp. output)}.
\]
\end{definition}
According to \cref{def_dual_hg},
for a weighted oriented hypergraph $H=(V, E, W_0^H, W_1^H)$,
$\mathcal{I}(H^T)=\mathcal{I}(H)^T$,
and the dual weighted oriented hypergraph is 
$H'=(V', E', W_0^{H'}, W_1^{H'})$,
with
$W_{0}^{H'}=W_{1}^{H}$,
$W_{1}^{H'}=W_{0}^H$.

\section{Cubical complexes}\label{appdx_cub_cplx}
In this section, we recall the basics of cubical complexes. There are many good sources for this, see e.g., \cite{MR3025945} for one written for a TDA audience.

\begin{definition}
An \emph{elementary interval} is defined as a unit interval $[k, k+1]\in \mathbb{R}$,
or a degenerate interval $[k, k]$.
\end{definition}

\begin{definition}
A $q$-\emph{cube} is defined as the product of $q$ unit intervals: $I\coloneqq \prod_{i=1}^q I_q$.
And we say the dimension of $I$ is $q$.
\end{definition}
By this definition,
$0$-cubes,
$1$-cubes,
$2$-cubes and $3$-cubes are vertices, edges, squares, and $3D$ cubes (voxels) respectively.

\begin{definition}
Let $I$, $J$ be two cubes.
We say $I$ is a \emph{face} of $J$ if and only if $I\subset J$.
\end{definition}

\begin{definition}
A subset $\KK\subset \mathbb{R}^q$ is called a \emph{cubical complex} of dimension $q$ if it is a collection of cubes of dimension at most $q$.
And $\KK$ is closed under taking faces and intersections.
\end{definition}

\begin{definition}
Let $\KK$ be a cubical complex.
The chain group $C_q^\KK$ is defined as the linear space generated by all $q$-cubes in $\KK$ over $\mathbb{R}$. 
\end{definition}

\begin{figure}[H]
    \centering
    \begin{tikzpicture}
  
    \draw[thick] (0,0) -- (2,0) -- (2,2) -- (0,2) -- cycle;
    \fill[yellow!40] (0,0) rectangle (2,2);
    
    \fill (0, 0) circle (2pt);  
    \fill (2, 0) circle (2pt);  
    \fill (2, 2) circle (2pt); 
    \fill (0, 2) circle (2pt);  

    \node[below left] at (0, 0) {$(m, n)$};
    \node[below right] at (2, 0) {$(m+1, n)$};
    \node at (2.4, 2.4) {$(m+1, n+1)$};
    \node at (-0.3, 2.4) {$(m, n+1)$};

    \node at (1, 1) {\Huge $I$};

\end{tikzpicture}
    \caption{A $2$-cube $I$}
    \label{fig_2_cube}
\end{figure}

\begin{definition}
We define the boundary map $\partial_q: C_{q}^\KK\to C_{q-1}^\KK$ as follows:
\begin{enumerate}
    \item For a $0$-cube $[k, k]$,
\[
\partial_{0}([k, k])\coloneqq 0,
\]
 \item for a $1$-cube $[k, k+1]$,
\[
\partial_1([k, k+1])\coloneqq [k+1, k+1] - [k, k],
\]
\item for a general $q$-cube $I=I_1\times \cdots I_q$
with $q>1$,
\[
\partial_q (I)\coloneqq
\partial_1(I_1)\times I_2\times \cdots, \times I_q
-
I_1 \times \partial_{q-1} (I_2\times \cdots \times I_q).
\]
\end{enumerate}
\end{definition}

As an illustration,
consider a simple $2$-cube $I=[m, m+1]\times [n, n+1]\subset \mathbb{R}^2$ as displayed in \cref{fig_2_cube},
the boundary of $I$ is: 
\[
\partial_2(I)
=
[m+1, m+1]\times [n, n+1] - [m, m]\times [n, n+1] 
-
[m, m+1]\times [n+1, n+1]
+
[m, m+1]\times [n, n].
\]

\section{Golub-Kahan-Lanczos method}
\label{appdx_GKL}
In this section we briefly review the Golub-Kahan-Lanczos method (GKL) to compute singular values of an $m\times n$ matrix $A$,
mainly following the exposition in \cite[Section 6.3.3]{MR1792141}.
Let $A\in M_{m\times n}(\mathbb{R})$,
and without loss of generality
we suppose that $m\geq n$; if $m<n$, we consider $A^T$ instead.
We refer to the eigenvectors of the symmetric matrix $A^TA$ (resp. $AA^T$) as \emph{right singular vectors} (resp. \emph{left singular vectors}).
Our goal is to bidiagonalize the matrix $A$,
i.e,
to find a matrix factorization
\[
U^TAV
=
B
=
\begin{bmatrix}
    \alpha_1 & \beta_1 & & & & \\
    & \alpha_2 & \beta_2 & \\
    & & \alpha_3 & \beta_3 \\
    & & & \ddots & \ddots \\
    & & & & \alpha_{n-1} & \beta_{n-1}\\
    & & & & & \alpha_n
\end{bmatrix}.
\]
Here $V$ is an $n\times n$ orthogonal matrix and $U$ is an $n\times m$ matrix orthogonal matrix whose column vectors are the left singular vectors associated with $A$,
and all $\alpha$'s and $\beta$'s are real.
We denote by 
$U=
\begin{bmatrix}
u_1 & u_2 & \cdots & u_n
\end{bmatrix}
$,
$V=
\begin{bmatrix}
v_1 & v_2 & \cdots & v_n
\end{bmatrix}
$,
and observe that
\[
\begin{aligned}
&A
\begin{bmatrix}
v_1 & v_2 & \cdots & v_n
\end{bmatrix}
=
\begin{bmatrix}
u_1 & u_2 & \cdots & u_n
\end{bmatrix}
B,\\
&A^T\begin{bmatrix}
u_1 & u_2 & \cdots & u_n
\end{bmatrix}
=
\begin{bmatrix}
v_1 & v_2 & \cdots & v_n
\end{bmatrix}
B^T.
\end{aligned}
\]
Computing the $j$th column on both sides,
we obtain that
\[
\begin{aligned}
&Av_j=\beta_{j-1}u_{j-1}+\alpha_ju_j,\\
&A^Tu_j = \alpha_jv_j+\beta_jv_{j+1},
\end{aligned}
\]
or equivalently,
\[
\begin{aligned}
&\alpha_ju_j=Av_j-\beta_{j-1}u_{j-1},
\quad \beta_0=0,\\
&\beta_{j}v_{j+1}=A^Tu_j-\alpha_jv_j.
\end{aligned}
\]
Since the columns of $U$ and $V$ are normalized,
we also have
\[
\begin{aligned}
&\alpha_j=||Av_j-\beta_{j-1}u_{j-1}||_2,\\
&\beta_j=||A^*u_j-\alpha_jv_j||_2.
\end{aligned}
\]
We can then compute the singular values of the bidiagonal matrix $B$ by QR algorithm,
according to \cite[Section 6.2]{MR1792141}; 
it costs $O(n^2)$.

In practical applications,
when the size of $A$ is very large,
it is usually not necessary to compute the entire bidiagonal matrix $B$,
instead,
we can iterate $k$ steps to obtain a $k\times k$ bidiagonal matrix $B_k$,
whose singular values can be treated as an approximation to those of $A$. Typically, the algorithm will converge faster to larger singular values.
We summarize the GKL algorithm in \cref{alg:GKL}. 

\begin{algorithm}[H]
 \caption{\textsc{GKL}}
 \label{alg:GKL}
 \begin{algorithmic}[1]
 \State \textbf{Input:} $A$, $k$
 \Comment{$k\leq n$}
 \State \textbf{Output:} $B_k$
 \Comment{$B$ is a bidiagonal matrix}
 
\State Randomly choose $v_1$ such that $||v_1||_2=1$ 
\State $\beta_0\gets 0$ 
\For{$j=1$ to $k$}
\State $u_j\gets Av_j - \beta_{j-1}u_{j-1}$
\State $\alpha_j\gets ||u_j||_2$
\State $u_j \gets u_j/\alpha_j$
\State $v_{j+1} \gets A^Tu_j-\alpha_jv_j$
\State $\beta_j\gets ||v_{j+1}||_2$
\State $v_{j+1}\gets v_{j+1}/\beta_j$
\EndFor

\end{algorithmic}
\end{algorithm}

\begin{proposition}
\label{prop:GKL_speed}
Let $A\in M_{m\times n}(\mathbb{R})$ be a non-branching matrix.
The GKL bidiagonalization of $A$ costs $O(k(m+n))$ if iterating $k$ times.
\end{proposition}
\begin{proof}
Consider the \cref{alg:GKL},
since $A$ is non-branching,
the non-zero entries of $A$ is bounded by $O(m)$,
hence both of the computations of $Av_j$ and $A^Tu_j$ cost $O(m)$.
On the other hand,
the computation of $u_j$ and $\alpha_j$ cost $O(m)$ since $u_j$ is $m$-dimensional,
similarly the computation of $v_j$ and $\beta_j$ cost $O(n)$ since $v_j$ is $n$-dimensional.
Hence in each iteration step it costs $O(m+n)$,
and in total it costs $O(k(m+n))$.
\end{proof}

\begin{proposition}
\label{prop:GKL_speed_2}
Let $A\in M_{m\times n}(\mathbb{R})$ be a matrix whose each column contains exactly $q$ non-zero entries.
The GKL bidiagonalization of $A$ with $k$ iterations costs $O\left(k\times (q\times n+m)\right)$.
\end{proposition}

\begin{proof}
Since the matrix $A$ contains exactly $q$ non-zero entries per column,
it contains exact $qn$ non-zero entries.
Hence both of the computations of $Av_j$ and $A^Tu_j$ cost $O(qn)$.
On the other hand,
it costs $O(m+n)$ to compute $u_j$, $\alpha_j$, $v_{j}$ and $\beta_{j+1}$.
Iterating $k$ times,
the result follows.
\end{proof}

\section{Numerical linear algebra example}
\label{ex:numeric}
We then assume that $\KK$ and $\LL$ have the non-trivial weights as follows
\[
W_{2}^{\LL}
=
\begin{bmatrix}
3\epsmach & 0\\
0 & \epsmach
\end{bmatrix},
\quad
W_1^\KK
=
\begin{bmatrix}
2 & 0 & 0 & 0 & 0\\
0 & 1 & 0 & 0 & 0\\
0 & 0 & 2 & 0 & 0\\
0 & 0 & 0 & 1 & 0\\
0 & 0 & 0 & 0 & 2
\end{bmatrix}.
\]
According to \cref{thm:fastcomputation},
$W_{2}^{\LL, \KK}
=
\begin{bmatrix}
1.5\epsmach & 0\\
0 & \epsmach
\end{bmatrix}$ and 
\[
\begin{aligned}
\Delta_{1, \mathrm{up}}^{\KK, \LL}
&=
B_2^{\LL, \KK} W_{2}^{\LL, \KK} (B_2^{\LL, \KK})^T(W_{1}^{\KK})^{-1}\\
&=
\begin{bmatrix}
0.75\epsmach & 1.5\epsmach & -0.75\epsmach & -1.5\epsmach & 0\\
0.75\epsmach & 2.5\epsmach & -0.75\epsmach & -2.5\epsmach & -0.5\epsmach\\
-0.75\epsmach & -1.5\epsmach & 0.75\epsmach & 1.5\epsmach & 0\\
-0.75\epsmach & -2.5\epsmach & 0.75\epsmach & 2.5\epsmach & 0.5\epsmach\\
0 & -\epsmach & 0 & \epsmach & 0.5\epsmach\\
\end{bmatrix}.
\end{aligned}
\]
In a computer with machine epsilon being $\epsmach$,
the up persistent Laplacian $\Delta_{1, \mathrm{up}}^{\KK, \LL}$ will be identified 
as 
\[
\Delta_{1, \mathrm{up}}^{\KK, \LL}
=
\begin{bmatrix}
0 & 1.5\epsmach & 0 & -1.5\epsmach & 0\\
0 & 2.5\epsmach & 0 & -2.5\epsmach & 0\\
0 & -1.5\epsmach & 0 & 1.5\epsmach & 0\\
0 & -2.5\epsmach & 0 & 2.5\epsmach & 0\\
0 & -\epsmach & 0 & \epsmach & 0\\
\end{bmatrix},
\]
as a result,
we will obtain only one non-zero eigenvalue since
the rank of $\Delta_{1, \mathrm{up}}^{\KK, \LL}$ will become $1$,
however,
the real rank of $\Delta_{1, \mathrm{up}}^{\KK, \LL}$ is $2$.
In contrast to that,
the matrix $M = (W_1^{\KK})^{-1/2}B_{2}^{\LL, \KK} (W_{2}^{\LL, \KK})^{1/2}$ is computed as
\[
M=
\begin{bmatrix}
0.87\sqrt{\epsmach} & 0\\
1.22\sqrt{\epsmach} & -\sqrt{\epsmach}\\
-0.87\sqrt{\epsmach} & 0\\
-1.22\sqrt{\epsmach} & \sqrt{\epsmach}\\
0 & 0.71\sqrt{\epsmach}
\end{bmatrix},
\]
we notice that the absolute value of all the entries in $M$ are larger than $\epsmach$,
thus we can obtain two non-zero singular values of $M$.
\end{document}